\documentclass[reqno]{amsart}
\usepackage{latexsym,epsfig,amssymb,amsmath,amsthm,color,url}
\usepackage[comma,sort&compress]{natbib}
\usepackage{hyperref}

\allowdisplaybreaks 
 \numberwithin{equation}{section}
\bibfont{\footnotesize}
\newtheorem{thm}{Theorem}[section]
\newtheorem{lem}[thm]{Lemma}
\newtheorem{prop}[thm]{Proposition}
\newtheorem{cor}[thm]{Corollary}

\renewcommand{\theenumi}{\roman{enumi}}

\theoremstyle{definition}
\newtheorem{defn}[thm]{Definition}

\theoremstyle{remark}
\newtheorem{rem}[thm]{Remark}

\newcommand{\bC}{\boldsymbol{C}}

\newcommand{\bD}{\boldsymbol{D}}
\newcommand{\bF}{\boldsymbol{F}}
\newcommand{\bI}{\boldsymbol{I}}
\newcommand{\bP}{{\boldsymbol{P}}}
\newcommand{\bQ}{\boldsymbol{Q}}
\newcommand{\bq}{\boldsymbol{q}}
\newcommand{\bR}{{\boldsymbol{R}}}

\newcommand{\bT}{\boldsymbol{T}}
\newcommand{\bu}{\boldsymbol{u}}
\newcommand{\bW}{\boldsymbol{W}}

\newcommand{\bone}{\boldsymbol{1}}
\newcommand{\bzero}{\boldsymbol{0}}
\newcommand{\bpi}{\boldsymbol{\pi}}
\newcommand{\bxi}{\boldsymbol{\xi}}
\newcommand{\bXi}{\boldsymbol{\Xi}}
\newcommand{\bchi}{\boldsymbol{\chi}}

\newcommand{\bzeta}{\boldsymbol{\zeta}}

\newcommand{\PF}{Perron-Frobenius\ }

\DeclareMathOperator{\E}{E}
\DeclareMathOperator{\Oh}{O}

\begin{document}
\bibliographystyle{plainnat}

\begin{center}
\textbf{\Large Strong Laws for Urn Models with Balanced Replacement Matrices}

\bigskip

\textsc{Amites Dasgupta} and \textsc{Krishanu Maulik}\\
Stat-Math Unit\\
Indian Statistical Institute\\
203 B. T. Road\\
Kolkata 700108\\
India\\
Email: amites,krishanu@isical.ac.in

\bigskip
\bigskip

\textbf{Abstract}\\
\end{center}
We consider an urn model, whose replacement matrix has all entries nonnegative and is balanced, that is, has constant row sums. We obtain the rates of the counts of balls corresponding to each color for the strong laws to hold. The analysis requires a rearrangement of the colors in two steps. We first reduce the replacement matrix to a block upper triangular one, where the diagonal blocks are either irreducible or the scalar zero. The scalings for the color counts are then given inductively depending on the \PF eigenvalues of the irreducible diagonal blocks. In the second step of the rearrangement, the colors are further rearranged to reduce the block upper triangular replacement matrix to a canonical form. Under a further mild technical condition, we obtain the scalings and also identify the limits. We show that the limiting random variables corresponding to the counts of colors within a block are constant multiples of each other. We provide an easy-to-understand explicit formula for them as well. The model considered here contains the urn models with irreducible replacement matrix, as well as, the upper triangular one and several specific block upper triangular ones considered earlier in the literature and gives an exhaustive picture of the color counts in the general case with only possible restrictions that the replacement matrix is balanced and has nonnegative entries.

\medskip


\medskip



\pagebreak

\begin{section}{Introduction} \label{sec: intro}
Consider an urn with balls of $D$ colors. The colors will be labeled by natural numbers as $1, 2, \ldots, D$. We start with an initial configuration of balls of different colors, where count of each color is strictly positive and add up to one. Note that the term ``count'' is an abuse of notation and it need not be an integer, but can be any positive real number. The fact here and later that ``count'' may not be integers does not cause much problem, as these numbers are used to define certain selection probabilities only in the sequel. The word ``count'' allows us to use the more picturesque language of ``drawing a ball''. Let the row vector $\bC_0$, which we assume to be a probability vector with all components positive, denote the initial count of balls of each color. The composition of the urn evolves by adding balls of different colors at times $n=1, 2, 3, \ldots$ as follows. The evolution of the composition of the urn will be governed by a replacement matrix $\bR$.

Throughout this article, we shall assume that the replacement matrix $\bR = ((r_{ij}))$ is a $D \times D$ non-random balanced (that is, each row sum is same and hence, without loss of generality, one) matrix with nonnegative entries. Again note that the entries $r_{ij}$ need not be integers, but real numbers. Let $\bC_N$ denote the row vector of the counts of balls of each color after the $N$-th trial, $N= 1, 2, \ldots$. We describe the evolution of $\bC_N$ inductively. At the $N$-th trial, a ball is drawn (or a color is selected) at random from the urn with the current composition $\bC_{N-1}$, so that the $i$-th color appears with probability $C_{N-1, i}/N$, $ i=1, \ldots, D$.  If the $i$-th color appears, then, for $j=1, \ldots, D$, $r_{ij}$ balls of $j$-th color are added to the urn before the next draw, together with the drawn ball, that is $C_{N, j} = C_{N-1, j} + r_{ij}$, for $j=1, \ldots, D$, when $i$-th color appears in the $N$-th draw. It is of interest to study the stochastic behavior of $\bC_N$ as $N \to \infty$.

If the replacement matrix is balanced with the common row sum $1$ and has nonnegative entries, then it can be viewed as a transition matrix of a Markov chain on a finite state space of size $D$ and it will be meaningful to talk about the reducibility or irreducibility of the matrix. However, the notion of irreducibility can easily be extended to any matrix with all entries nonnegative, see, for example, Chapter~1.3 of \cite{seneta:2006book}.
\begin{defn}\label{def: irred}
A $D\times D$ matrix $\bR$ with all entries nonnegative is called \textit{irreducible} if for each $1\le i, j \le D$, there exists $n\equiv n(i,j)$ such that the $(i,j)$-th entry of $\bR^n$ is strictly positive. A matrix, which is not irreducible, will be called \textit{reducible}.
\end{defn}
\noindent Note that, for an irreducible matrix, there may not exist a common $n$ such that $\bR^n$ has all entries strictly positive. As an example, for the matrix $\left(\substack{0\,1\\1\,0}\right)$, the $(1,1)$-th entry of all odd powers and $(1,2)$-th entry of all even powers will be zero. Any irreducible matrix has a positive eigenvalue of algebraic multiplicity one, which is larger than or equal to all other eigenvalues in modulus. Such an eigenvalue is called the \textit{\PF} eigenvalue of the irreducible matrix. Since no other eigenvalue equals the \PF eigenvalue, which is real and positive, the \PF eigenvalue is strictly larger than the real part of any other eigenvalue. The left and the right eigenvectors corresponding to the \PF eigenvalue have all entries strictly positive. The \PF eigenvalue will be contained in the interval formed by the smallest and the largest row sum. The \PF eigenvalue will be in the interior of the interval unless the matrix is balanced. For a discussion on the \PF eigenvalues of irreducible matrices, we refer to Chapter~1.4 of \cite{seneta:2006book}.

In case the replacement matrix $\bR$ is irreducible, its \PF eigenvalue will be $1$, as it is balanced with common row sum $1$. Let $\bpi_R$ be the left eigenvector of $\bR$, normalized so that the sum of the coordinates is $1$, corresponding to the \PF eigenvalue. Then $\bpi_\bR$ is also the unique stationary distribution satisfying $\bpi_\bR \bR = \bpi_\bR$ and will have all coordinates strictly positive. Then \citep[see, for example,][]{gouet:1997} $\bC_N / (N+1) \to \bpi_R\ \text{almost surely}$. This strong law for the color counts in the irreducible case has been studied in more general setups and further strong/weak laws, central and functional central limit theorems for linear combinations of color counts are well known. We refer to \cite{bai:hu:1999} for the martingale approach and to \cite{janson:2004} for the branching process approach; both papers also contain detailed references to the literature.

However, when the replacement matrix is not irreducible or balanced, the balls of different colors may increase at different rates and strong/weak limits for $\bC_n$ are not known in full generality. \cite{flajolet:dumas:puyhaubert:2006, janson:2006, bose:dasgupta:maulik:2009a, bose:dasgupta:maulik:2009b} and the references in these papers contain some results in these directions which are relevant in this context. The case of upper triangular $\bR$ has been studied in these papers, sometimes under suitable assumptions.

Actually some strong laws are also available for more general case of reducible $\bR$. For example, the case of balanced block triangular $\bR$ with irreducible diagonal blocks have been identified in \cite{gouet:1997} as an important class among the reducible ones. The assumption of balanced rows leads to convenient application of martingale techniques. More precisely, let us assume that $\bR$ is balanced and upper block triangular with $K+1$ blocks, like
$$\begin{pmatrix}
\bQ_1 & \cdots & \cdots & \cdots & \cdots \\
\bzero & \bQ_2 &\cdots & \cdots & \cdots \\
\vdots & \ddots  & \ddots & \cdots & \cdots \\
\bzero & \cdots & \bzero & \bQ_K & \cdots \\
\bzero & \cdots & \bzero & \bzero & \bP
\end{pmatrix}$$
where $\bQ_1, \ldots, \bQ_K$ are irreducible (but not necessarily balanced) matrices with Perron-Frobenius eigenvalue less than 1 and $\bP$ is irreducible and obviously balanced with \PF eigenvalue 1. Let $\bpi_\bP$ satisfy $\bpi_\bP = \bpi_\bP \bP$. Then Proposition~4.3 of \cite{gouet:1997} says that for such an $\bR$, $\bC_N/(N+1) \rightarrow (\bzero, \bzero, \ldots , \bzero, \bpi_\bP)$ almost surely, that is the vectors of color counts corresponding to the first $K$ blocks are killed if scaled by $N$. This raises the question whether the rates of the vectors of color counts corresponding to the first $K$ blocks can be identified.

\cite{janson:2006} has studied a two color urn model with reducible replacement matrix. He took the matrix to be lower triangular, but did not put any further restrictions. The matrix was allowed to have different row sums, as well as, possibly negative entries. The asymptotic behavior of the color counts were discussed in details using branching process techniques.

The above two urn models inspired \cite{bose:dasgupta:maulik:2009b} to consider the urn model with triangular replacement matrix and, under some technical assumptions, the rates of individual color counts were identified. It is clear that the triangular model, where one deals with blocks of size one, is a special case of block triangular models with irreducible diagonal blocks. Section~1.2 of \cite{seneta:2006book} sketches an arrangement which reduces a matrix with nonnegative entries to a block lower triangular one. By further rearranging the states in a reverse order, the matrix can be made into a block upper triangular one. In fact, any balanced replacement matrix can be reduced to a block upper triangular one where the diagonal blocks are either irreducible or the scalar zero through a rearrangement of the colors. It should be stressed that the result is true for any matrix with nonnegative entries and the equality of row sums is not important. The states can be identified with colors. Note that any rearrangement of colors is same as a similarity transform by a permutation matrix. Here the nonzero irreducible diagonal block need not be balanced. The rearrangement is quite simple in nature, but we have not come across any detailed ready reference in the literature. So we quickly outline a proof of the rearrangement in the following lemma.
\begin{lem}\label{lem: rearrange}
Any matrix $\bR$, with all entries nonnegative, is similar to a block upper triangular matrix, whose diagonal blocks are either irreducible or the scalar zero, via a permutation matrix.
\end{lem}
\begin{proof}
As already explained, we shall explain the proof through a rearrangement of colors, which is equivalent to a similarity transform through a permutation matrix.

We shall say a color $i$ leads to a color $j$, if for some $n$, the $(i,j)$-th entry of $\bR^n$ is positive. The colors $i$ and $j$ are said to communicate if both $i$ leads to $j$ and conversely. The class of $i$ is defined as $\mathcal C_i = \{j: i \text{ communicates with } j\}$. Note that either $\mathcal C_i$ is empty or contains $i$. The colors with empty classes will be called lone colors. For two different colors $i$ and $j$, either their classes coincide or they are disjoint. Further note that the submatrices of $\bR$ corresponding to each nonempty class is irreducible. Next, make singleton classes of each lone color. The collection of all distinct classes (including the singleton classes of the lone color) forms a partition of the collection of all colors and they will form the required blocks after a permutation. A class $\mathcal C$ is said to lead to another class $\mathcal C'$ if some color in $\mathcal C$ leads to another color in $\mathcal C'$ and we shall write $\mathcal C \preceq \mathcal C'$. It is easy to see that ``$\preceq$'' is a well-defined, transitive and anti-symmetric relation and hence is a partial order on the collection of distinct classes. So the collection of distinct classes can be rearranged in a non-decreasing order. The corresponding rearrangement of colors will have the replacement matrix in the required block upper triangular form with zero or irreducible diagonal blocks. The diagonal blocks corresponding to nonempty classes of some color will give the irreducible ones, while the lone colors will give the scalar zero diagonal blocks.
\end{proof}

It should be noted that the eigenvalues together with multiplicities remain unchanged under similarity transforms. Also, as the similarity transform is done by a permutation matrix, this will result in the eigenvectors being rearranged correspondingly. In this article, without loss of generality, we only consider the case of balanced block triangular $\bR$ with scalar zero or irreducible diagonal blocks. Note that, for irreducible replacement matrix, we have only one block. A special case of two irreducible diagonal blocks, both of which are balanced, was treated in \cite{bose:dasgupta:maulik:2009a}. The strong law there is given in Proposition~4.2(iii) and the proof follows from the proof of Theorem~3.1(iv) of the same article. The proof essentially used the strong law for the irreducible case mentioned earlier along with the introduction of a stopping time. However, when the irreducible diagonal blocks are not balanced, we require new techniques to handle the strong convergence of the vectors of color counts corresponding to the diagonal blocks. This article presents these new techniques along with a simplification of earlier proofs using Kronecker's lemma. The limits are identified later in the article after a further rearrangement and an extra technical assumption is made. The limits involve suitably normalized left and right eigenvectors of the appropriate irreducible diagonal blocks corresponding to their \PF eigenvalues. The initial zero diagonal blocks identified after this rearrangement give a different type of limits.

As a consequence of the rearrangement mentioned in Lemma~\ref{lem: rearrange}, the $D\times D$ balanced, block upper triangular replacement matrix $\bR$ with nonnegative entries is assumed to have $K+1$ blocks, where the diagonal blocks are either irreducible or the scalar zero and none but last of which need to be balanced. The $k$-th block contains $d_k$ many colors with $d_1 + \cdots + d_{K+1} = D$. We shall denote the blocks by $\bQ_{kl}$, where $k, l = 1, 2, \ldots, K+1$. Thus, $\bQ_{kl}$ will be of dimension $d_k \times d_l$ and $\bQ_{kl}=0$, whenever $k>l$. We shall generally denote the diagonal block $\bQ_{kk}$ by $\bQ_k$.

Let $\bchi_N$ be the row vector called the incidence vector whose $m$-th entry will be $1$ and all other entries $0$, if $m$-th color is drawn at the $N$-th draw. The subvectors of $\bC_N$ and $\bchi_N$ corresponding to $k$-th block of colors will be denoted by $\bC_N^{(k)}$ and $\bchi_N^{(k)}$ respectively. Let $\mathcal{F}_N$ denote the $\sigma$-field generated by the collection of random vectors $\{\bchi_1, \ldots, \bchi_N\}$. We have the following evolution equation:
\begin{equation} \label{eq: evolution}
\bC_{N+1} = \bC_N + \bchi_{N+1} \bR.
\end{equation}
We shall show that the rates of growth of the color count subvector will be constant in each block and the rate for the $k$-th block will be of the form $N^{\alpha_k} \log^{\beta_k}N$.
\begin{defn} \label{def: rate pair}
If the color count subvector corresponding to the $k$-th block grows at the rate $N^{\alpha_k} \log^{\beta_k}N$, that is, $\bC_N^{(k)}/(N^{\alpha_k} \log^{\beta_k}N)$ converges almost surely and in $L^2$, then we shall denote the rate by the \textit{rate pair} $(\alpha_k,\beta_k)$.

The ordering of the rates of growth induces an ordering on the rate pairs, which is the lexicographical ordering, that is, the color count subvector of the $k$-th block grows at a rate faster than that of the $k^\prime$-th block if and only if either $\alpha_k>\alpha_{k^\prime}$ or $\alpha_k=\alpha_{k^\prime}$ and $\beta_k>\beta_{k^\prime}$.
\end{defn}

One of the goals of this article is to obtain the rate pairs of the count subvectors corresponding to all the blocks, which we do in Theorem~\ref{thm: algo}. The rate pairs depend on the \PF eigenvalues of the diagonal block matrices, whenever it is irreducible. This introduces another important notion of this article.
\begin{defn}\label{def: char}
For a square matrix $\bQ$ with nonnegative entries, which is either irreducible or zero, we define its \textit{character} $\mu$ as the \PF eigenvalue, if $\bQ$ is irreducible, and as $0$, if $\bQ=\bzero$.

For an upper triangular matrix $\bR$ formed by nonnegative entries with $K+1$ diagonal blocks $\{\bQ_k\}_{1\le k\le K+1}$, which are either irreducible or zero matrices, the \textit{character} of the $k$-th block will be denoted by $\mu_k$.
\end{defn}

We shall show that the rate pair of the first block $(\alpha_1,\beta_1)=(\mu_1,0)$. The rate pairs of the later blocks will be defined inductively. The rate pair of the $k$-th block will be determined by the (lexicographically) largest among the rate pairs $(\alpha_m,\beta_m)$ with $m=1,\ldots,k-1$ satisfying $\bQ_{mk}\neq\bzero$. If the largest such pair is denoted by $(\alpha,\beta)$, then $\alpha_k = \max\{\alpha,\mu_k\}$ and $\beta_k$ will be $\beta$, $\beta+1$ or $0$ according as $\alpha$ is greater than, equal to or less than $\mu_k$. This shows the crucial role played by the character in determining the rates.

In Section~\ref{sec: auxiliary}, we bring in some notations and prove some results which are useful for obtaining the rates of growth of the color counts. Using these results, we prove the rates, as defined above, in Section~\ref{sec: algo}. We introduce further notions, the rearrangement to the increasing order and the assumption~\eqref{assmp} in Section~\ref{sec: notations}. Finally, in Section~\ref{sec: main}, we identify the limits for the replacement matrix in the increasing order under the extra technical assumption~\eqref{assmp}. Suitably normalized left and right eigenvectors corresponding to the \PF eigenvalues of the irreducible diagonal blocks play an important role in identifying the limits. Thus, we obtain the rate of the color count subvectors for all urn models with only possible restrictions of nonnegativity of the entries and the balanced condition on the matrix. We identify the limits as well, but under the extra technical assumption~\eqref{assmp}. In the process, we identify the very important role played by the characters of all the diagonal blocks and suitably normalized left and right eigenvectors of certain irreducible diagonal blocks corresponding to their \PF eigenvalues.
\end{section}

\begin{section}{Notations and some auxiliary results} \label{sec: auxiliary}
We begin this section by recalling the notion of Jordan canonical form of a matrix. We need to introduce the square matrix $\bF$ for that purpose. The order of the matrix will be clear from the context. The matrix $\bF$ will have all entries zero except the ones in the diagonal just above the main diagonal, namely,
$$f_{ij} =
\begin{cases}
  1, &\text{if $j=i+1$,}\\
  0, &\text{otherwise.}
\end{cases}$$
If the order of the matrix is $1$, then the corresponding scalar is defined as $0$. The matrix $\bF$ is nilpotent. In particular, if $\bF$ has order $d$, then $\bF^d = \bzero$. Further, for any $1 \le i < d$, $\bF^i$ has all entries zero except the $i$-th diagonal above the main one, which has all entries one. If $\nu$ is an eigenvalue of a matrix $\bR$, then, define the Jordan block corresponding to $\nu$ as $\bD_\nu = \nu \bI + \bF$. We also have a matrix $\bXi_\nu$ of full column rank, whose columns are Jordan vectors corresponding to $\nu$. In fact, the first column of $\bXi_\nu$ is a right eigenvector of $\bR$ corresponding to $\nu$ and $\bXi_\nu$ satisfies $\bR \bXi_\nu = \bXi_\nu \bD_\nu$. In the Jordan decomposition of $\bR$, given by $\bR \bXi = \bD \bXi$, the matrix $\bD$ is block diagonal with diagonal blocks given by $\bD_\nu$ corresponding to some eigenvalue $\nu$. The total number of blocks (possibly of different dimensions), that an eigenvalue $\nu$ contributes to $\bD$ equals its geometric multiplicity and the sum of the dimensions of the blocks corresponding to $\nu$ equals its algebraic multiplicity. The matrix $\bXi$ is obtained by concatenating the matrices $\bXi_\nu$ in the corresponding order.

If $z$ is a non-zero complex number, we denote by $\bT_z$ an upper triangular matrix, which has $(i,j)$-th entry is $z^{-(j-i+1)}$, for $j\ge i$. As $\bF$ is a nilpotent matrix, we have
$$\bT_z = \frac1z \left[\bI + \sum_{i=1}^\infty \left( \frac1z \bF \right)^i \right] = \frac1z \left(\bI - \frac1z \bF\right)^{-1} = (zI-\bF)^{-1}.$$ Now, if $\lambda$ is a positive number larger than the absolute value of any eigenvalue of a matrix $\bR$, then $(\lambda \bI - \bR)$ is invertible. If $\nu$ is an eigenvalue of $\bR$ with the corresponding Jordan decomposition $\bR \bXi_\nu = \bXi_\nu \bD_\nu = \bXi_\nu (\nu \bI + \bF)$, then we have $(\lambda \bI - \bR) \bXi_\nu = \bXi_\nu ((\lambda-\nu) \bI - \bF) = \bXi_\nu \bT_{\lambda-\nu}^{-1}$ and hence
\begin{equation} \label{eq: T}
\bXi_\nu \bT_{\lambda-\nu} = (\lambda \bI - \bR)^{-1} \bXi_\nu.
\end{equation}

We further use the following notation, defined for all complex numbers $z$, except for the negative integers,
$$\Pi_N (z) = \prod_{n=0}^{N-1} \left( 1 + \frac{z}{n+1} \right),$$ which satisfies Euler's formula for the Gamma function,
\begin{equation} \label{eq: Euler}
\Pi_N (z) \sim N^z / \Gamma (z+1).
\end{equation}

For a vector $\bxi$, the vectors $|\bxi|^2$ and $\bxi^2$ will denote the vectors whose entries are squares of the moduli and squares of the entries of the vector $\bxi$ respectively. For two real vectors $\bxi$ and $\bzeta$ of same dimension, inequalities like $\bxi \le \bzeta$ will correspond to the inequalities for each coordinate.

For a complex number $z$, we shall denote its real and imaginary parts by $\Re z$ and $\Im z$ respectively.

We are now ready to do the analysis for obtaining the rates of the color counts for each block. The presence of diagonal blocks as matrices with possibly complex eigenvalues introduces additional complications compared to the triangular case. Also in the block triangular case, it is not wise to study the individual color counts directly. We consider the linear combinations of color counts in each block with respect to eigenvectors and Jordan vectors. Before obtaining the rates for each color count, we state some auxiliary results in this section, which will be useful later in proving the rates. The first result concerns a simple observation regarding (possible complex valued) martingales, which follows from two simple applications of Kronecker's lemma.
\begin{lem}\label{lem: mart}
Let $\{M_N\}$ be a $($possibly complex valued\,$)$ martingale with the martingale difference sequence $\Delta M_N = M_{N+1} - M_N$ satisfying $\E[|\Delta M_N|^2] = \Oh(c_N)$ for some sequence of positive numbers $\{c_N\}$. If for some other sequence of positive numbers $\{a_N\}$, which diverges to infinity, we have $\sum_{n=1}^\infty (c_{n-1}/a_n^2) < \infty$, then $M_N/a_N\to0$ almost surely, as well as, in $L^2$.
\end{lem}
\begin{proof}
Observe that $\E[|\Delta M_N/a_N|^2] = \Oh(c_N/a_N^2)$, which is summable. Thus, $\sum_{n=0}^{N-1} \Delta M_n/a_n$ forms an $L^2$-bounded martingale, which converges almost surely. Then both the real and the imaginary parts of this martingale will also converge almost surely. Further, as $a_N\to\infty$, using Kronecker's lemma, both $\Re M_N / a_N$ and $\Im M_N / a_N$ converge to $0$ almost surely. Thus $M_N/a_N\to0$ almost surely.

Further, since $a_N$ diverges to infinity and $\{c_{N-1}/a_N^2\}$ is summable, we have, again using Kronecker's lemma, $\sum_{n=1}^N c_{n-1} / a_N^2$ converges to zero. Further, $\E[|M_N|^2]=M_0^2+\sum_{n=0}^{N-1} \E[|\Delta M_n|^2] = \Oh\left(\sum_{n=0}^{N-1} c_n\right)$. Thus, $\E[|M_N|^2]/a_N^2\to0$.
\end{proof}

For the second result, we consider a block upper triangular replacement matrix with three blocks.
\begin{lem} \label{lem: three blk}
Consider an urn model with the replacement matrix
\begin{equation} \label{model: three block}
\bR =
\begin{pmatrix}
\bQ_{1} &\bQ_{12} &\bq_{13}\\
\bzero &\bQ_2 &\bq_{23}\\
\bzero &\bzero &1
\end{pmatrix},
\end{equation}
which is balanced and has all entries nonnegative, with $d_1$, $d_2$ and $1$ colors in three blocks respectively. None of the submatrices need to be balanced and, except for $\bQ_2$, none of the submatrices need to be irreducible either. However, $\bQ_2$ is assumed to be irreducible with the \PF eigenvalue $\mu$ and the corresponding right eigenvector $\bzeta$. Let $\nu$ be another eigenvalue of $\bQ_2$ with Jordan decomposition given by $\bQ_2 \bXi_\nu = \bXi_\nu \bD_\nu$. The rows of $\bQ_{12}$ are $\{\bq_l\}_{1\le l\le d_1}$, some of which may be the zero row vectors. The color count vector and its subvectors are denoted as before.

Also assume that there exists $\alpha \ge \mu$ and an integer $\beta \ge 0$, such that, for all $l=1,\ldots,d_1$, satisfying $\bq_l\neq\bzero$, we have,
\begin{equation} \label{eq: conv assmpn}
\frac{C_{N,l}}{N^\alpha \log^\beta N} \to u_l \quad \text{almost surely and in $L^2$},
\end{equation}
where $u_l$ is nonnegative, but can be random. Further assume that
$$\bC_N^{(2)} \bzeta / (N^\alpha \log^\beta N) \quad \text{converges almost surely and in $L^2$}$$ to a nondegenerate random variable. Then
\begin{equation} \label{eq: three blk conv}
\frac{\bC_N^{(2)}}{N^\alpha \log^\beta N} \bXi_\nu \to \sum_{\substack{1\le l\le d_1\\l: \bq_l\neq\bzero}} u_l \bq_l \bXi_\nu \bT_{\alpha-\nu} \quad \text{almost surely and in $L^2$}.
\end{equation}
\end{lem}
\begin{proof}
Let $\mathcal F_N$ denote the $\sigma$-field generated by the collection $\{\bchi_1,\ldots,\bchi_N\}$ as before. The incidence vector and its subvector are defined as before. Let $\bxi_1, \ldots, \bxi_t$ be the columns of $\bXi_\nu$ with $t\ge 1$. Define $\bxi_0 = \bzero$. By the definition of $\bT_{\lambda-\nu}$, it is equivalent to prove that, for $i = 1, \ldots, t$,
\begin{equation} \label{eq: ind hyp lem}
\frac{\bC_N^{(2)} \bxi_i}{N^\alpha \log^\beta N} \to \sum_{\substack{1\le l\le d_1\\l: \bq_l\neq\bzero}} u_l \bq_l \left( \frac1{(\alpha-\nu)^{i}} \bxi_1 + \frac1{(\alpha-\nu)^{i-1}} \bxi_2 + \cdots + \frac1{(\alpha-\nu)} \bxi_i \right) \quad \text{almost surely and in $L^2$},
\end{equation}
which we shall do by induction.

Now, from the Jordan decomposition, for $i=1,\ldots,t$, we have $\bQ_2 \bxi_i = \bxi_{i-1} + \nu \bxi_i = \widetilde\bxi$. Further, using the evolution equation~\eqref{eq: evolution}, we have
$$\bC_{N}^{(2)} \bxi_i = \bC_{N-1}^{(2)} \bxi_i + \sum_{\substack{1\le l\le d_1\\l: \bq_l\neq\bzero}} \chi_{N,l} \bq_l \bxi_i + \bchi_{N}^{(2)} \widetilde\bxi.$$ From this, we obtain the martingale
\begin{equation} \label{eq: mg three blk nonPF}
M_N = \frac{\bC_N^{(2)} \bxi_i}{\Pi_N (\nu)} - \sum_{n=0}^{N-1} \frac1{(n+1) \Pi_{n+1} (\nu)} \left( \sum_{\substack{1\le l\le d_1\\l: \bq_l \neq \bzero}} C_{n,l} \bq_l \bxi_i + \bC_n^{(2)} \bxi_{i-1} \right)
\end{equation}
having martingale difference
$$\Delta M_N = \frac1{\Pi_{N+1} (\nu)} \left[ \sum_{\substack{1\le l\le d_1\\l: \bq_l\neq\bzero}} \left( \chi_{N+1,l} - \frac1{N+1} C_{N,l} \right) \bq_l \bxi_i + \left( \bchi_{N+1}^{(2)} - \frac1{N+1} \bC_N^{(2)} \right) \widetilde\bxi \right].$$ Since,$\bzeta$ is a right eigenvector corresponding to the \PF eigenvalue of $\bQ_2$, it has all coordinates positive and hence, for some $c>0$, we have $|\widetilde \bxi|^2 \le c \bzeta$. Hence, using Euler's formula~\eqref{eq: Euler} and the fact that at most one of $\chi_{N+1,l}$ for $l=1,\ldots,d_1$ and $\bchi_{N+1}^{(2)}$ can be nonzero simultaneously, we have,
$$\E\left[ |\Delta M_N|^2\right] = \Oh \Bigg( N^{-(1+2\Re\nu)} \left( \E \left[ \textstyle{\sum_{\substack{1\le l\le d_1\\l: \bq_l\neq\bzero}}} C_{N,l} |\bq_l \bxi_i|^2 \right] + \E \left[ \bC_N^{(2)} \bzeta \right]\right) \Bigg).$$ Hence, by the assumptions made on the rates of convergence of $C_{N,l}$ for $l=1,\ldots,d_1$ with $\bq_l\neq\bzero$, and $\bC_N^{(2)} \bzeta$, we obtain $\E[ |\Delta M_N|^2 ] = \Oh ( {\log^\beta N}/{N^{1+2 \Re \nu - \alpha}} )$. Next, we apply Lemma~\ref{lem: mart} with $c_N=\log^\beta N / N^{1+2\Re\nu-\alpha}$ and $a_N=N^{\alpha-\Re\nu}\log^\beta N$. Since $\alpha\ge\mu>0$ and $\Re\nu<\mu\le\alpha$, Lemma~\ref{lem: mart} applies and $M_N/(N^{\alpha-\Re\nu}\log^\beta N)$ and hence $M_N/(N^{\alpha-\nu} \log^\beta N)$ converges to zero almost surely and in $L^2$.

Thus, from Euler's formula~\eqref{eq: Euler} and the definition of the martingale~$M_N$ in~\eqref{eq: mg three blk nonPF}, we have,
\begin{align*}
\lim_{N\to\infty} \frac{\bC_N^{(2)} \bxi_i}{N^\alpha \log^\beta N} = &\lim_{N\to\infty} \frac1{N^{\alpha-\nu} \log^\beta N} \sum_{\substack{1\le l\le d_1\\l: \bq_l\neq\bzero}} \sum_{n=0}^{N-1} \frac{(n+1)^\nu}{\Pi_{n+1}(\nu) \Gamma(\nu+1)} \frac{\log^\beta (n+2)}{(n+1)^{1+\nu-\alpha}} \frac{C_{n,l} \bq_l \bxi_{i}}{(n+1)^\alpha \log^\beta (n+2)}\\
& \quad + \lim_{N\to\infty} \frac1{N^{\alpha-\nu} \log^\beta N} \sum_{n=0}^{N-1} \frac{(n+1)^\nu}{\Pi_{n+1}(\nu) \Gamma(\nu+1)} \frac{\log^\beta (n+2)}{(n+1)^{1+\nu-\alpha}} \frac{\bC_n^{(2)} \bxi_{i-1}}{(n+1)^\alpha \log^\beta (n+2)},
\end{align*}
where the limits are both in almost sure and in $L^2$ sense and we use~\eqref{eq: conv assmpn} in the last step.

Since $\alpha\ge\mu>\Re\nu$, the first term above simplifies to $\frac1{\alpha-\nu} {\sum}^\prime u_l \bq_l \bxi_i$, where the sum is over all $l=1,\ldots,d_1$, such that $\bq_l\neq\bzero$. If for some $i\ge1$, $\lim_{N\to\infty} \bC_N^{(2)} \bxi_{i-1} / (N^\alpha \log^\beta N)$ exists almost surely and in $L^2$, then
$$\lim_{N\to\infty} \frac{\bC_N^{(2)} \bxi_i}{N^\alpha \log^\beta N} = \frac1{\alpha-\nu} \sum_{\substack{1\le l\le d_1\\l: \bq_l\neq\bzero}} u_l \bq_l \bxi_i + \frac1{\alpha-\nu} \lim_{N\to\infty} \frac{\bC_N^{(2)} \bxi_{i-1}}{N^\alpha \log^\beta N}$$
almost surely and in $L^2$. For $i=1$, since $\bxi_0=\bzero$, we immediately have~\eqref{eq: ind hyp lem}. Assuming the induction hypothesis for $i-1$,~\eqref{eq: ind hyp lem} can now easily be extended to $i$ as well.
\end{proof}

\begin{rem} \label{rem: lim zero}
If $\bq_l=\bzero$ for all $l=1,\ldots,d_1$, the argument of the above proof still goes through with the obvious modification that any sum over the indices $l=1,\ldots,d_1$ such that $\bq_l\neq\bzero$ will be zero, and the limit in~\eqref{eq: three blk conv} will also be zero.
\end{rem}

Finally, we obtain some moment bounds for the color counts in the block upper triangular model, as reduced by Lemma~\ref{lem: rearrange}. We first obtain the expectation of the linear combination of the count vector of a block.

\begin{lem} \label{lem: first mom}
Consider an urn model with balanced, block upper triangular replacement matrix $\bR$ formed of nonnegative entries, where the $k$-th diagonal block $\bQ_k$ is either irreducible or the scalar zero with the character $\mu_k$. If $\bQ_k$ is irreducible, let $\bzeta$ be a right eigenvector corresponding to the \PF eigenvalue, which is also the character, $\mu_k$. If $\bQ_k$ is the scalar zero, let $\bzeta$ be the scalar one. Then
\begin{equation} \label{eq: first mom}
\E \left[ \bC_{N}^{(k)} \bzeta \right] = \Pi_{N} (\mu_k) \left( \bC_0^{(k)} \bzeta + \sum_{\substack{1\le m\le k-1\\m: \bQ_{mk}\neq\bzero}} \sum_{n=0}^{N-1} \frac1{(n+1) \Pi_{n+1} (\mu_k)} \E\left[ \bC_n^{(m)} \bQ_{mk} \bzeta \right] \right).
\end{equation}
\end{lem}
\begin{proof}
From the evolution equation~\eqref{eq: evolution}, we get $\bC_{N}^{(k)} \bzeta = \bC_{N-1}^{(k)} \bzeta + \sum_{m=1}^{k} \bchi_{N}^{(m)} \bQ_{mk} \bzeta$. Taking conditional expectation, we have
\begin{equation} \label{eq: cond exp}
\E \left[ \bC_{N}^{(k)} \bzeta | \mathcal{F}_{N-1} \right] = \left( 1 + \frac{\mu_k}{N} \right) \bC_{N-1}^{(k)} \bzeta + \sum_{\substack{1\le m\le k-1\\m: \bQ_{mk}\neq\bzero}} \frac1{N} \bC_{N-1}^{(m)} \bQ_{mk} \bzeta.
\end{equation}
Taking further expectation and iterating, the result follows.
\end{proof}

Next, we define a martingale and obtain a bound on the square moments of the martingale difference.
\begin{lem} \label{lem: mg diff}
Consider an urn model with balanced, block upper triangular replacement matrix $\bR$ formed of nonnegative entries, where the $k$-th diagonal block $\bQ_k$ is either irreducible or the scalar zero with the character $\mu_k$. If $\bQ_k$ is irreducible, let $\bzeta$ be a right eigenvector corresponding to the \PF eigenvalue, which is also the character, $\mu_k$. If $\bQ_k$ is the scalar zero, let $\bzeta$ be the scalar one. Then
\begin{equation} \label{def: mg}
M_N = \frac{\bC_N^{(k)} \bzeta}{\Pi_N(\mu_k)} - \sum_{\substack{1\le m\le k-1\\m: \bQ_{mk}\neq\bzero}} \sum_{n=0}^{N-1} \frac{\bC_n^{(m)} \bQ_{mk} \bzeta}{(n+1) \Pi_{n+1} (\mu_k)}
\end{equation}
is a martingale and, for the martingale difference $\Delta M_N = M_{N+1} - M_N$, we have, for some constant $c>0$,
\begin{equation} \label{eq: mg diff}
\E \left[ (\Delta M_N)^2 \right] \le \frac{c}{(N+1) (\Pi_{N+1} (\mu_k))^2} \sum_{\substack{1\le m\le k\\m: \bQ_{mk}\neq\bzero}} \E \left[ \bC_N^{(m)} \bone \right],
\end{equation}
where $\bQ_{kk}=\bQ_k$. When $\bQ_k$ is the scalar zero or equivalently $\mu_k=0$, the above bound~\eqref{eq: mg diff} simplifies to
\begin{equation} \label{eq: mg diff zero}
\E \left[ (\Delta M_N)^2 \right] \le \frac{c}{(N+1)} \sum_{\substack{1\le m\le k-1\\m: \bQ_{mk}\neq\bzero}} \E \left[ \bC_N^{(m)} \bone \right].
\end{equation}
\end{lem}
\begin{proof}
The fact that $M_N$ is a martingale follows from the expression for the conditional expectation in~\eqref{eq: cond exp}. We also have
$$\Delta M_N = \frac1{\Pi_{N+1} (\mu_k)} \left[ \mu_k \left( \bchi_{N+1}^{(k)} - \frac1{N+1} \bC_N^{(k)} \right) + \sum_{\substack{1\le m\le k-1\\m: \bQ_{mk}\neq\bzero}} \left( \bchi_{N+1}^{(m)} - \frac1{N+1} \bC_N^{(m)} \right) \bQ_{mk} \right] \bzeta.$$ Since $\bchi_{N+1}^{(m)}$ cannot be nonzero simultaneously for two distinct values of $m$, taking conditional expectation and ignoring the negative terms, we have,
\begin{equation} \label{eq: cond exp sq mg diff}
\E \left[ (\Delta M_N)^2 | \mathcal{F}_N \right] \le \frac1{(N+1) (\Pi_{N+1} (\mu_k))^2} \left[ \mu_k^2 \bC_N^{(k)} \bzeta^2  + \sum_{\substack{1\le m\le k-1\\m: \bQ_{mk}\neq\bzero}} \bC_N^{(m)} \left( \bQ_{mk} \bzeta \right)^2 \right].
\end{equation}
Since $\bone$ has all coordinates equal to one, and hence, positive, we have, for some constant $c>0$, $\bzeta^2 \le c \bone$ and for $m=1, \ldots, k-1$ satisfying $\bQ_{mk}\neq\bzero$, $\left( \bQ_{mk} \bzeta \right)^2 < c \bone$. Putting these bounds and the fact that $\mu_k^2 \le 1$ in~\eqref{eq: cond exp sq mg diff} and taking expectation,~\eqref{eq: mg diff} follows.

When $\bQ_k$ is the scalar zero or equivalently $\mu_k=0$, then~\eqref{eq: mg diff zero} follows from the simple observations that $\Pi_{N+1}(0)=1$ and the first term within the bracket in~\eqref{eq: cond exp sq mg diff} is absent.
\end{proof}

\begin{rem} \label{rem: miss terms}
If $\bQ_{mk}=\bzero$ for all $m=1,\ldots,k-1$, then the results and the arguments of Lemmas~\ref{lem: first mom} and~\ref{lem: mg diff} will still go through with obvious modifications. The last sum within the bracket on the right side of~\eqref{eq: first mom} and the last term in the definition of the martingale in~\eqref{def: mg} will be absent. The sum on the right side of~\eqref{eq: mg diff} will reduce to $\E[\bC_N^{(k)}\bone]$. If further $\bQ_k=\bzero$, then $\bC_N^{(k)}=\bC_0^{(k)}$ for all $N$ and the martingale defined in~\eqref{def: mg} will be a constant, as $\mu_k=0$ as well. This will give $\E[(\Delta M_N)^2]=0$ in~\eqref{eq: mg diff zero}.
\end{rem}
\end{section}

\begin{section}{Rates of color counts} \label{sec: algo}
We are now ready to give an inductive method to obtain the rates of the color count subvectors corresponding to each block.
\begin{thm}\label{thm: algo}
Consider an urn model with a balanced, block upper triangular replacement matrix $\bR$ formed by nonnegative entries and with blocks $\{\bQ_{ml}\}_{1\le m,l\le K+1}$, where the diagonal blocks $\bQ_{kk}=\bQ_k$ are either the scalar zero or an irreducible matrix, for $k=1,\ldots, K+1$. Let the characters of the diagonal blocks be $\{\mu_k\}_{1\le k\le K+1}$. The color count vector and its subvectors are defined as before. Then, for $k=1,\ldots,K+1$, there exists nonnegative real numbers $\alpha_k$ and nonnegative integers $\beta_k$, such that $(\alpha_k,\beta_k)$ are the rate pairs for $\bC_N^{(k)}$, that is, $\bC_N^{(k)}/(N^{\alpha_k} \log^{\beta_k} N)$ converges almost surely, as well as, in $L^2$. The pairs $\{(\alpha_k,\beta_k)\}_{1\le k\le K+1}$ are defined inductively as follows: For $k=1$, $\alpha_1=\mu_1$ and $\beta_1=0$. Having defined $(\alpha_1,\beta_1),\ldots,(\alpha_{k-1},\beta_{k-1})$, let $(\alpha,\beta)$ be the (lexicographically) largest rate pair in the set $\{(\alpha_m,\beta_m): 1\le m\le k-1, \bQ_{m,k}\neq\bzero\}$. If the set is empty, declare $(\alpha,\beta)=(-\infty,0)$. Then, we define
\begin{align*}
\alpha_{k} &= \max\{\alpha,\mu_{k}\} \intertext{and}
\beta_{k} &=
\begin{cases}
0, &\text{if $\mu_{k}>\alpha$,}\\
\beta+1, &\text{if $\mu_{k}=\alpha$,}\\
\beta, &\text{if $\mu_{k}<\alpha$.}
\end{cases}
\end{align*}
\end{thm}
\begin{proof}
We use induction on the number of blocks $k$. For the case $k=1$, if $\mu_1=0$, then the first color count remains constant and hence converges without scaling. If $\mu_1>0$, then $\bQ_1$ is irreducible with \PF eigenvalue $\mu_1$ and a corresponding right eigenvector $\bzeta$. Since $\bzeta$ has all coordinates positive, choose $c>0$ such that $\bzeta^2\le c\,\bzeta$. It is then easy to see that $M_N^{\prime}=\bC_N^{(1)} \bzeta/\Pi_N(\mu_1)$ is a martingale with the martingale difference
$$\Delta M_N^{\prime}=\frac{\mu_1}{\Pi_{N+1}(\mu_1)}\left(\bchi_{N+1}^{(1)}-\frac1{N+1}\bC_N^{(1)}\right)\bzeta.$$ Since $0<\mu_1\le1$, we get, using Euler's formula~\eqref{eq: Euler}, $\E[\left(\Delta M_N^{\prime}\right)^2]\le {c}\E\left[M_N^{\prime}\right]/({(N+1) \Pi_{N+1}(\mu_1)}) = \Oh\left(N^{-(1+\mu_1)}\right)$, which is summable. Hence $M_N^{\prime}$ is an $L^2$-bounded martingale, which converges to a nondegenerate random variable almost surely and in $L^2$, and thus, by Euler's formula~\eqref{eq: Euler}, $\bC_N^{(1)}\bzeta/N^{\mu_1}$ also converges to a nondegenerate random variable $Y_1$ almost surely and in $L^2$.

Next consider any eigenvalue $\nu$ of $\bQ_1$ other than the \PF one, $\mu_1$. Let the corresponding Jordan decomposition be $\bQ_1 (\bxi_1:\cdots:\bxi_t) = (\bxi_1:\cdots:\bxi_t) \bD_\nu$, for some $t\ge1$. Note that $\bQ_1\bxi_i = \bxi_{i-1} + \nu\bxi_i$, for $i=1,\ldots,t$, where $\bxi_0=\bzero$. We define the martingale
$$M_N^{\prime\prime}= \frac{\bC_N^{(1)}\bxi_i}{\Pi_N(\nu)} - \sum_{n=0}^{N-1} \frac{\bC_n^{(1)}\bxi_{i-1}}{(n+1)\Pi_{n+1}(\nu)}$$ as in the proof of Lemma~\ref{lem: three blk} and arguing similarly, we get $\E[\left(\Delta M_N^{\prime\prime}\right)^{2}]=\Oh(N^{-(1+2\Re\nu-\mu_1)})$. Then, again applying Lemma~\ref{lem: mart} with $c_N=N^{-(1+2\Re\nu-\mu_1)}$ and $a_N=N^{\mu_1-\Re\nu}$ and arguing as in the proof of Lemma~\ref{lem: three blk}, we have $M_N^{\prime\prime}/N^{\mu_1-\nu}\to0$ almost surely and in $L^2$. Again, simplifying using Euler's formula~\eqref{eq: Euler}, we have,
$$\lim_{N\to\infty} \frac1{N^{\mu_1}} \bC_N^{(1)}\bxi_i =\lim_{N\to\infty} \frac1{N^{\mu_1-\nu}} \sum_{n=0}^{N-1} \frac{(n+1)^\nu}{\Pi_{n+1}(\nu) \Gamma(\nu+1)} \frac1{(n+1)^{1-\mu_1+\nu}}\frac{\bC_n^{(1)}\bxi_{i-1}}{(n+1)^{\mu_1}},$$
where the limits are both in almost sure and in $L^2$ sense. As $\bxi_0=\bzero$, the limit is zero for $i=1$ and then inductively, it can be shown that the limits are zero for all $i=1,\ldots,t$. This gives $\bC_N^{(1)} \bXi_\nu / N^{\mu_1} \to \bzero$ almost surely and in $L^2$. Finally, consider $\bR \bXi = \bXi \bD$, the Jordan decomposition of $\bR$, where $\bzeta$ is the first column of $\bXi$. Then
$\bC_N^{(1)} \bXi / N^{\mu_1} \to Y_1 (1,0,\ldots,0)$ and hence
\begin{equation} \label{eq: lim one}
\bC_N^{(1)} / N^{\mu_1} \to Y_1 \bpi \quad \text{almost surely and in $L^2$,}
\end{equation}
where $\bpi$ is the first row of $\bXi^{-1}$ and is a left eigenvector (normalized so that $\bpi\bzeta=1$) of $\bQ_1$ corresponding to the \PF eigenvalue of $\bQ_1$. Hence $\bpi$ has all coordinates positive. This shows the rate pair $(\alpha_1,\beta_1)=(\mu_1,0)=(\mu_1,\kappa_1)$. This technique of handling nonzero characters will be repeated for the later blocks as well. For a block with a nonzero character, which is then the \PF eigenvalue of the corresponding irreducible diagonal block, we first find out the rate of the linear combination of the corresponding count subvector with respect to a right eigenvector corresponding to the \PF eigenvalue. The limit will be a nondegenerate random variable, which will be a function of previous such random variables, unless the block is the leading block of its cluster with the order of the corresponding leading character zero. We then obtain the limits of the linear combinations corresponding to the Jordan vectors as well with the same rate and combine them to get the final result for the count subvector.

Assume that the rate pairs have been obtained for the first $k-1$ blocks. We define $(\alpha,\beta)$ and $(\alpha_k,\beta_k)$ as in the statement of the theorem and show that $(\alpha_k,\beta_k)$ is the required rate pair for the $k$-th block. If $\alpha=-\infty$, then $\bQ_{mk}=\bzero$ for all $m<k$. If we further have $\mu_k=0$, that is, $\bQ_k=\bzero$, then $\bC_{N}^{(k)}=\bC_0^{(k)}$ for all $N$ and the rate pair will be $(\alpha_k,\beta_k)=(0,0)$ as required. So assume either $\alpha\neq-\infty$ or $\mu_k\neq0$. Equivalently we have
\begin{equation} \label{eq: nonzero}
\bQ_{mk}\neq\bzero \quad \text{for some $m=1,\ldots,k$.}
\end{equation}

First consider the case $\mu_k=0$, that is, $\bQ_k$ is the scalar zero. Hence, from~\eqref{eq: nonzero}, we have $\bQ_{mk}\neq\bzero$ for some $m=1,\ldots,k-1$. Then the set $\{(\alpha_m,\beta_m): 1\le m\le k-1, \bQ_{mk}\neq\bzero\}$ is nonempty and $\alpha_k$ is a nonnegative real and $\beta_k$ is a nonnegative integer. Define the martingale $M_N$ as in Lemma~\ref{lem: mg diff} with $\bzeta$ as the scalar one, since $\mu_k=0$. Then using~\eqref{eq: mg diff zero} and the choice of $(\alpha,\beta)$, we have
\begin{equation} \label{eq: diff bd zero}
\E\left[(\Delta M_N)^2\right] = \Oh\left( N^{-(1-\alpha)} \log^\beta N\right).
\end{equation}
Then we apply Lemma~\ref{lem: mart} with $c_N=N^{-(1-\alpha)} \log^\beta N$ and $a_N=N^{\alpha_k} \log^{\beta_k} N$. Since $\mu_k=0$ and $\alpha$ is nonnegative, we have only two possibilities, $\alpha>\mu_k=0$ and $\alpha=\mu_k=0$. Observe that
$$\frac{c_{N-1}}{a_N^2} \sim
\begin{cases}
{\displaystyle \frac1{N\log^{\beta+2}N}}, &\text{if $\alpha=\mu_k=0$, or equivalently, $\alpha_k=0$ and $\beta_k=\beta+1$,}\\
{\displaystyle \frac1{N^{1+\alpha}\log^{\beta}N}}, &\text{if $\alpha>\mu_k=0$, or equivalently, $\alpha_k=\alpha$ and $\beta_k=\beta$}
\end{cases}$$
and thus $\sum_N (c_{N-1}/a_N^2)<\infty$. So, from Lemma~\ref{lem: mart}, we have $M_N/(N^{\alpha_k} \log^{\beta_k} N)\to0$ almost surely and in $L^2$. Simplifying using Euler's formula~\eqref{eq: Euler}, the definition of the martingale, the choice of $(\alpha, \beta)$ and the facts that $\bzeta$ is the scalar one and $N^{\alpha_k} \log^{\beta_k}N\to\infty$, we have
\begin{equation} \label{eq: lim zero}
\lim_{N\to\infty} \frac{\bC_N^{(k)}}{N^{\alpha_k} \log^{\beta_k} N} =
\begin{cases}
\displaystyle{\sum_{\substack{1\le m\le k-1\\m:\bQ_{mk}\neq\bzero}}} \frac1{\beta+1} \displaystyle{\lim_{N\to\infty}} \frac{\bC_N^{(m)}}{\log^\beta N} \bQ_{mk}, &\text{if $\alpha=\mu_k=0$,}\\
\displaystyle{\sum_{\substack{1\le m\le k-1\\m:\bQ_{mk}\neq\bzero}}} \frac1{\alpha} \displaystyle{\lim_{N\to\infty}} \frac{\bC_N^{(m)}}{N^\alpha \log^\beta N} \bQ_{mk}, &\text{if $\alpha>\mu_k=0$,}
\end{cases}
\end{equation}
where the limits are almost sure, as well as, in $L^2$. Also, all the limits are nonnegative and as $(\alpha, \beta)$ is the largest rate pair, at least one of them is nondegenerate. So the limit above is nondegenerate and we have the rate pair $(\alpha_k,\beta_k)$ as suggested in the statement of the theorem. The limit in~\eqref{eq: lim zero} will be denoted by $Y_k$.

Next consider the case $\mu_k>0$. Then $\bQ_k$ is irreducible with the \PF eigenvalue $\mu_k$. Choose $\bzeta$ as a right eigenvector of $\bQ_k$ corresponding to $\mu_k$. Since $\bzeta$ has all coordinates positive, choose $c>0$, such that $\bone\le c\bzeta$.

If $\alpha=-\infty$, that is, $\bQ_{mk}=\bzero$ for all $m=1,\ldots,k-1$, we shall apply Lemmas~\ref{lem: first mom} and~\ref{lem: mg diff} keeping Remark~\ref{rem: miss terms} in mind. In this case $\alpha_k=\mu_k$ and $\beta_k=0$. From Lemma~\ref{lem: first mom}, we have $\E\left[\bC_N^{(k)}\bzeta\right] = \Oh \left( N^{\mu_k} \right)$. Defining the martingale $M_N$ as in Lemma~\ref{lem: mg diff}, we have $\E[(\Delta M_N)^2]=\Oh(N^{-(1+\mu_k)})$, which is summable. Thus, $M_N$ and hence, by Euler's formula~\eqref{eq: Euler}, $\bC_N^{(k)}\bzeta/N^{\mu_k} = \bC_N^{(k)}\bzeta/(N^{\alpha_k}\log^{\beta_k}N)$ converges almost surely and in $L^2$ to nondegenerate limits and the limit of the second sequence will be denoted as $Y_k$.

Finally consider the possibility that $\alpha\neq-\infty$ and $\mu_k>0$. Then $\alpha$ is a nonnegative real and $\beta$ is a nonnegative integer. From Lemma~\ref{lem: first mom}, we have,
$$\E \left[ \frac1{\Pi_N (\mu_k)} \bC_{N}^{(k)} \bzeta \right] = \bC_0^{(k)} \bzeta + \sum_{\substack{1\le m\le k-1\\m:\bQ_{mk}\neq\bzero}} \sum_{n=0}^{N-1} \frac{\log^{\beta} (n+2)}{(n+1)^{1+\mu_k-\alpha}} \frac{(n+1)^{\mu_k}}{\Pi_{n+1} (\mu_k)} \E\left[ \frac{\bC_n^{(m)} \bQ_{mk} \bzeta}{(n+1)^{\alpha} \log^{\beta} (n+2)} \right].$$
By the choice of $(\alpha,\beta)$, the expectations in the sum of the right side above are bounded. Hence, using Euler's formula~\eqref{eq: Euler}, we have
\begin{align*}
\E\left[\bC_N^{(k)}\bone\right] \le c\E\left[\bC_N^{(k)}\bzeta\right] &=
\begin{cases}
\Oh(N^{\mu_k}), &\text{if $\mu_k>\alpha$,}\\
\Oh(N^{\mu_k} \log^{\beta+1}N), &\text{if $\mu_k=\alpha$,}\\
\Oh(N^{\alpha} \log^\beta N), &\text{if $\mu_k<\alpha$}
\end{cases}\\
&=\Oh(N^{\alpha_k} \log^{\beta_k} N).
\end{align*}
Observe that the rate pair $(\alpha_k, \beta_k)$ is lexicographically larger than or equal to $(\alpha, \beta)$ and hence $(\alpha_k,\beta_k)$ gives the highest rate, giving,
$$\sum_{\substack{1\le m\le k\\m:\bQ_{mk}\neq\bzero}} \E \left[ {\bC_N^{(m)}\bone} \right] = \Oh\left(N^{\alpha_k} \log^{\beta_k} N\right).$$ Then define the martingale $M_N$ as in Lemma~\ref{lem: mg diff} and we have from~\eqref{eq: mg diff},
$$\E\left[(\Delta M_N)^2\right] = \Oh \left( N^{-(1+2\mu_k-\alpha_k)} \log^{\beta_k}N \right).$$ We then apply Lemma~\ref{lem: mart} with $c_N = N^{-(1+2\mu_k-\alpha_k)} \log^{\beta_k}N$ and $a_N = N^{\alpha_k - \mu_k} \log^{\beta_k} N$. Observe that $c_{N-1}/a_N^2 \sim N^{-(1+\alpha_k)} \log^{-\beta_k}N$. Now $\alpha_k=\max\{\alpha,\mu_k\}\ge\mu_k>0$ and hence $\sum_N (c_{N-1}/a_N^2) <\infty$. So, again from Lemma~\ref{lem: mart}, we have $M_N/(N^{\alpha_k-\mu_k}\log^{\beta_k}N)\to0$ almost surely and in $L^2$. Further simplifying using Euler's formula~\eqref{eq: Euler}, the definition of the martingale and the choice of $(\alpha, \beta)$, we have
\begin{multline} \label{eq: lim positive}
\lim_{N\to\infty} \frac{\bC_N^{(k)}\bzeta}{N^{\alpha_k} \log^{\beta_k} N}\\ = \lim_{N\to\infty} \frac1{N^{\alpha_k-\mu_k}\log^{\beta_k}N} \sum_{\substack{1\le m\le k-1\\m:\bQ_{mk}\neq\bzero}} \sum_{n=0}^{N-1} \frac{\log^\beta (n+2)}{(n+1)^{1+\mu_k-\alpha}} \frac{(n+1)^{\mu_k}}{\Pi_{n+1}(\mu_k) \Gamma(\mu_k+1)} \frac{\bC_n^{(m)} \bQ_{mk}\bzeta}{(n+1)^\alpha \log^\beta (n+2)},
\end{multline}
where the limits are almost sure, as well as, in $L^2$. Again, all the limits are nonnegative and as $(\alpha, \beta)$ is the largest rate pair, at least one of them is nondegenerate. So the limit above is nondegenerate and will be denoted by $Y_k$. If $\mu_k<\alpha=\alpha_k$ and $\bu_m$ denotes the almost sure and $L^2$ limit of $\bC_N^{(m)}/(N^\alpha\log^\beta N)$ for $m=1,\ldots,k-1$ with $\bQ_{mk}\neq\bzero$, then the limit in~\eqref{eq: lim positive} can be further simplified to
\begin{equation} \label{eq: lim positive nonlead}
\lim_{N\to\infty} \frac{\bC_N^{(k)}\bzeta}{N^{\alpha_k} \log^{\beta_k} N} = \frac1{\alpha_k-\mu_k}\sum_{\substack{1\le m\le k-1\\m:\bQ_{mk}\neq\bzero}} \bu_m \bQ_{mk} \bzeta = \sum_{\substack{1\le m\le k-1\\m:\bQ_{mk}\neq\bzero}} \bu_m \bQ_{mk} (\alpha_k\bI-\bQ_k)^{-1} \bzeta.
\end{equation}
Thus, ${\bC_N^{(k)}\bzeta}/({N^{\alpha_k} \log^{\beta_k} N})$ converges to a nondegenerate limit $Y_k$ both almost surely and in $L^2$ and for $\mu_k<\alpha$, the limit is given by the right side of~\eqref{eq: lim positive nonlead}.

Next consider an eigenvalue $\nu$ of $\bQ_k$ other than the \PF eigenvalue with the corresponding Jordan decomposition $\bQ_k \bXi_\nu = \bXi_\nu \bD_\nu$. Then club all the colors after the $k$-th block into a single one and make the first $k-1$ blocks into one group and the $k$-th block into another. This gives us the replacement matrix in the form~\eqref{model: three block}. Also, by the choice of $(\alpha,\beta)$,~\eqref{eq: conv assmpn} holds with the rate pair $(\alpha,\beta)$ and hence $(\alpha_k, \beta_k)$. (Note that $0$ is a possible limit in~\eqref{eq: conv assmpn}.) We also have that ${\bC_N^{(k)}\bzeta}/({N^{\alpha_k} \log^{\beta_k} N})$ converges to a nondegenerate random variable $Y_k$ almost everywhere, as well as in $L^2$. Thus, by Lemma~\ref{lem: three blk}, we have
\begin{equation} \label{eq: lim positive nonPF}
\frac{\bC_N^{(k)}\bXi_\nu}{N^{\alpha_k} \log^{\beta_k} N}\to \sum_{\substack{1\le l\le d_1+\cdots+d_{k-1}\\l: \bq_l\neq\bzero}} u_l \bq_l \bXi_\nu \bT_{\alpha_k-\nu} \quad \text{almost surely and in $L^2$,}
\end{equation}
where $u_l$ is the almost sure and $L^2$ limit of ${\bC_{N,l}}/({N^{\alpha_k} \log^{\beta_k} N})$ for any index $l$ allowed in the sum. If $\alpha=-\infty$, that is, $\bq_l=\bzero$ for all $l=1,\ldots,d_1+\cdots+d_{k-1}$, then, by Remark~\ref{rem: lim zero}, the limit in~\eqref{eq: lim positive nonPF} still holds with the interpretation that the limit is zero.

If $\mu_k\ge\alpha$ (this includes the case $\alpha=-\infty$), then observe that the rate pair $(\alpha_k,\beta_k)$ gives a higher rate than $(\alpha,\beta)$ and thus the limits $u_m$ in~\eqref{eq: lim positive nonPF} are all zero, which gives ${\bC_N^{(k)}\bXi_\nu}/({N^{\alpha_k} \log^{\beta_k} N})\to\bzero$ almost surely and in $L^2$. If $\mu_k<\alpha$, then $\mu_k<\alpha=\alpha_k$, then from~\eqref{eq: T}, we have $\bXi_\nu \bT_{\alpha_k-\nu} = (\alpha_k \bI - \bQ_k)^{-1} \bXi_\nu$. Also, observe that, by the induction hypothesis, the rates are same within a block. Thus, if any index $l$ is included in the sum on the right side of~\eqref{eq: lim positive nonPF}, we can include any other index $l^\prime$ in the same block with $\bq_{l^\prime}=\bzero$, as ${\bC_{N,l^\prime}}/({N^{\alpha_k} \log^{\beta_k} N})$ will also converge to $u_{l^\prime}$ almost surely and in $L^2$, but will not contribute anything extra. Further, for the $m$-th block, the limit vector $\bu_m$ in~\eqref{eq: lim positive nonlead} consists of such $u_l$'s only. So we can rewrite~\eqref{eq: lim positive nonPF} as
$$\lim_{N\to\infty} \frac{\bC_N^{(k)}\bXi_\nu}{N^{\alpha_k} \log^{\beta_k} N} = \sum_{\substack{1\le m\le k-1\\m:\bQ_{mk}\neq\bzero}} \bu_m \bQ_{mk} (\alpha_k\bI-\bQ_k)^{-1} \bXi_\nu.$$

Finally, consider $\bQ_k \bXi = \bXi \bD$, the Jordan decomposition of $\bQ_k$, where the first column of $\bXi$ is $\bzeta$. Then, we have
$$\lim_{N\to\infty} \frac{\bC_N^{(k)}\bXi}{N^{\alpha_k} \log^{\beta_k} N} =
\begin{cases}
Y_k (1,0,\ldots,0), &\text{if $\mu_k\ge\alpha$,}\\
\sum_{\substack{1\le m\le k-1\\m:\bQ_{mk}\neq\bzero}} \bu_m \bQ_{mk} (\alpha_k\bI-\bQ_k)^{-1} \bXi, &\text{if $\mu_k<\alpha$.}
\end{cases}$$
Observing, as in the case $k=1$, that $\bXi^{-1}$ has the first row as $\bpi$, a left eigenvector (normalized so that $\bpi\bzeta=1$) of $\bQ_k$ corresponding to the \PF eigenvalue, which has all coordinates positive, we conclude that
\begin{equation} \label{eq: lim blk k}
\lim_{N\to\infty} \frac{\bC_N^{(k)}}{N^{\alpha_k} \log^{\beta_k} N} =
\begin{cases}
Y_k \bpi, &\text{if $\mu_k\ge\alpha$,}\\
{\displaystyle \sum_{\substack{1\le m\le k-1\\m:\bQ_{mk}\neq\bzero}} \bu_m \bQ_{mk} (\alpha_k\bI-\bQ_k)^{-1}}, &\text{if $\mu_k<\alpha$.}
\end{cases}
\end{equation}
This shows that $(\alpha_k,\beta_k)$ is the rate pair for the $k$-th block and completes the induction step.
\end{proof}
\end{section}

\begin{section}{Rearrangement to the increasing order} \label{sec: notations}
To identify the limits when the color counts are scaled as in Theorem~\ref{thm: algo}, we need to first rearrange the colors further and reduce the replacement matrix to the increasing form (see Definition~\ref{def: incr}) and make a technical assumption~\eqref{assmp}. The rearrangement to the increasing order is an extension of Proposition~2.4 of \cite{bose:dasgupta:maulik:2009b}. Before going into the rearrangement, we need to introduce some notions in analogy to Section~2 of \cite{bose:dasgupta:maulik:2009b}.

Note that, in Lemma~\ref{lem: rearrange}, we have required the zero diagonal blocks to be scalar. However, we do not impose any such condition here. We shall require zero diagonal blocks of higher dimensions for the rearrangement to the increasing order in Lemma~\ref{lem: rearrange2}. Also, if $\bQ$ is irreducible, then the character $\mu$ is its \PF eigenvalue, and hence all its eigenvalues are smaller than or equal to $\mu$ in modulus. If $\bQ=\bzero$, then all its eigenvalues are zero. Thus, in either case, all the eigenvalues of $\bQ$ is smaller than or equal to its character in modulus.
\begin{defn}\label{def: max}
For a block upper triangular matrix $\bR$ formed by nonnegative entries with $K+1$ characters $\{\mu_k\}_{1\le k\le K+1}$, let $1\le i_1<i_2<\cdots<i_J<i_{J+1}\le K+1$ be the \textit{indices of the running maxima of the characters}, that is, $\mu_1=\mu_{i_1}\le\mu_{i_2}\le\cdots\le \mu_{i_J}\le\mu_{i_{J+1}}$ and for $i_j<k<i_{j+1}$ with $j=1,\ldots,J$ or $i_j<k\le K+1$ with $j=J+1$, we have $\mu_k<\mu_{i_j}$.
\end{defn}
\noindent Since the replacement matrix $\bR$ is assumed to be balanced, we necessarily have $\bQ_{K+1}$ is balanced and $\mu_{K+1}=1$. Also, the row sums of all other diagonal blocks are less than or equal to one and hence $\mu_k\le 1$ for $k=1,\ldots,K$. So we necessarily have $i_{J+1}=K+1$.
\begin{defn}
In a balanced block upper triangular matrix $\bR$ formed by the nonnegative entries with characters $\{\mu_k\}_{1\le k\le K+1}$ and their indices of running maxima $\{i_j\}_{1\le j\le J+1}$, the blocks indexed by $i_j,i_j+1,\ldots,i_{j+1}-1$ form the $j$-th \textit{cluster} for $j=1,\ldots,J$ and the $i_{J+1}$-th block alone forms the $(J+1)$-th \textit{cluster}. For $j=1,\ldots,J+1$, the $i_j$-th block is called the \textit{leading block} of the $j$-th cluster.

We shall define the \textit{leading character} as $\lambda_j=\mu_{i_j}$ for $j=1,\ldots,J+1$ and the \textit{order} of the leading character as, for $j=1,\ldots,J+1$,
\begin{equation} \label{def: nu}
\kappa_j = \# \{j' < j : \lambda_{j'} = \lambda _j\} = \# \{k < j : \mu_k = \mu_{i_j} \},
\end{equation}
which counts the number of earlier occurrences of the character of a leading block.
\end{defn}

The concepts of clusters and the leading blocks are to be viewed in comparison to the notions of the blocks and the leading colors in \cite{bose:dasgupta:maulik:2009b}. Since the characters are nonnegative, if $\lambda_j=0$ for some $j=1,\ldots,J+1$, then we must have $\lambda_1=\cdots=\lambda_j=0$, $i_1=1,\ldots,i_j=j$ and $\kappa_1=0,\ldots,\kappa_j=j-1$. Thus, if there are zero diagonal blocks at the beginning, all of them will form clusters of size one and these are the only leading zero diagonal blocks.

Any balanced block upper triangular matrix $\bR$, formed by nonnegative entries and reduced to the form where all diagonal blocks are either irreducible matrices or the scalars zero, can be further reduced to a form which we describe next.
\begin{defn}\label{def: incr}
A balanced block upper triangular matrix $\bR$, formed by nonnegative entries and blocks $\{\bQ_{kl}\}_{1\le k,l\le K+1}$, which are either irreducible or zero matrices, with indices of running maxima of characters $\{i_j\}_{j=1}^{J+1}$ and leading characters $\{\lambda_j\}_{j=1}^{J+1}$, is said to be in the \textit{increasing order} if
\begin{enumerate}
\item All non-leading zero diagonal blocks are scalar. \label{def: incr scalar}
\item For index $k$ of any non-leading block, that is, $i_j < k < i_{(j+1)}$ with $j = 1, 2, \ldots, J$, we have
$$\sum_{m=i_j}^{k-1} \bQ_{mk} \neq \bzero.$$ \label{def: incr nonlead}
\item If $\lambda_j=0$ and $\kappa_j>0$ for some $j=2,\ldots,J$, then each column of the submatrix $\bQ_{i_j-1,i_j}$ must be nonzero. \label{def: incr zero}
\end{enumerate}
\end{defn}
\noindent The condition~\eqref{def: incr nonlead} holds for any non-leading block. In fact, the condition~\eqref{def: incr nonlead} in the above definition implies that at least one entry of the submatrices $\bQ_{mk}$ for $m=i_j,\ldots,k-1$ to be nonzero. However, the condition~\eqref{def: incr zero} extends the requirement to the leading blocks as well, when the leading character is zero and has a positive order. Further, in this case, we not only have the submatrices as nonzero, we indeed have each of the columns of the submatrices as nonzero. Note that the condition~\eqref{def: incr zero} in the above definition is vacuous if $\lambda_1>0$.

It was shown in Proposition~2.4 of \cite{bose:dasgupta:maulik:2009b} that any balanced upper triangular matrix can be reduced to another block upper triangular matrix satisfying the condition~\eqref{def: incr nonlead} of Definition~\ref{def: incr} by a similarity transform using permutation matrices. As in Section~\ref{sec: intro}, the similarity transform using a permutation matrix will be viewed as a rearrangement of the states, will not change the eigenvalues with their multiplicities and will rearrange any eigenvector accordingly. We extend that rearrangement in the following result.
\begin{lem}\label{lem: rearrange2}
Any balanced block triangular matrix $\bR$, with all entries nonnegative and diagonal blocks either an irreducible matrix or the scalar zero, is similar to a matrix in increasing order via a permutation matrix.
\end{lem}
\begin{proof}
As in the proof of Lemma~\ref{lem: rearrange}, we shall describe the similarity transform by a permutation matrix through a rearrangement of the states. We do it in two steps. In the first step, we obtain a matrix satisfying condition~\eqref{def: incr nonlead} in Definition~\ref{def: incr} alone. The rearrangement is very similar to Proposition~2.4 of \cite{bose:dasgupta:maulik:2009b}, where we replace the diagonal entries by the characters of the blocks, the leading colors by the leading blocks and the blocks by the clusters. We do not repeat the proof here. By an abuse of notation, we shall use the same set of notations for the rearranged matrix. Note that all the zero diagonal blocks remain scalar after this rearrangement.

If the first diagonal block is nonzero and hence all the leading characters are nonzero, the condition~\eqref{def: incr zero} in Definition~\ref{def: incr} becomes vacuously true and the proof is complete. Thus, without loss of generality, assume that the first $M(>0)$ diagonal states are zero. They are the only leading blocks, which are zero. In the second step, we shall rearrange the colors in these $M$ states to satisfy the condition~\eqref{def: incr zero} in Definition~\ref{def: incr}. Thus the condition~\eqref{def: incr nonlead} in Definition~\ref{def: incr} will remain unaffected. Also none of the non-leading zero diagonal blocks will be affected and they will remain scalar satisfying the condition~\eqref{def: incr scalar} in Definition~\ref{def: incr}. Note that this second step of rearrangement may coalesce some of the initial zero diagonal states and hence the overall number of blocks may reduce. We shall proceed by induction. The $m$-th step will produce a zero diagonal block (not necessarily scalar), which is also a cluster, with the (leading) character zero and order $m-1$. We shall show that after $m$-th induction step with $m\ge 1$, the rearranged matrix is block upper triangular, the first $m$ blocks, and thus clusters, have $\{d_k\}_{1\le k\le m}$ states and satisfy the condition~\eqref{def: incr zero} in Definition~\ref{def: incr} and for the remaining $M-(d_1+\cdots+d_m)$ states, the columns have at least one entry corresponding to the blocks after the first $m-1$ ones, that is, at least one entry with index more than $d_1+\cdots+d_{m-1}$, as nonzero.

Observe that since $\lambda_1=0$, the first column must be the zero vector. We first bring all states (which includes the first state), whose columns are zero vectors, to the front and declare that they constitute the first block, as well as the cluster. The order of the states within the block is not important. Note that this rearrangement maintains the block upper triangular structure. The condition on the remaining states among the first $M$ ones also holds, since all zero columns have been collected. Again, by abuse of notations, we shall use the same set of notations for the matrix thus rearranged. If the size of the first block $d_1=M$, we are done. Otherwise, assume that we have obtained $m\ge 1$ blocks of sizes $\{d_k\}_{1\le k\le m}$, which, except the first one, satisfy the condition~\eqref{def: incr zero} in Definition~\ref{def: incr} and for the remaining $M-(d_1+\cdots+d_m)$ states, the columns have at least one of the entries, indexed more than $d_1+\cdots+d_{m-1}$, as nonzero. Also, the rearranged matrix is block upper triangular. As before, by an abuse, we retain the notations. If $d_1+\cdots+d_m=M$, we are done.

Otherwise, we consider the remaining $M-(d_1+\cdots+d_m)$ states among the first $M$ ones. The columns corresponding to all these states have at least one entry with index more than $d_1+\cdots+d_{m-1}$ as nonzero, by the induction hypothesis. Thus, for any column having all entries with index more than $d_1+\cdots+d_m$ as zero, one of the entries indexed $d_1+\cdots+d_{m-1}+1$ through $d_1+\cdots+d_m$ must be nonzero. The $(d_1+\cdots+d_m+1)$-th column satisfies this condition and thus the set of the states satisfying this condition is nonempty. We collect all the states satisfying this condition, that is, with the corresponding columns having entries with index more than $d_1+\cdots+d_m$ as zero and bring them forward, after the $m$-th block, ahead of the rest, to form the $(m+1)$-th block, as well as the cluster. The order of the states within the block is again not important. Since the columns have all entries with index more than $d_1+\ldots+d_m$ zero, the block upper triangular structure is retained. Also, after rearrangement, this block satisfies the condition~\eqref{def: incr zero} in Definition~\ref{def: incr}, as, corresponding to each state, the column has at least one entry corresponding to $m$-th block, indexed $d_1+\cdots+d_{m-1}+1$ through $d_1+\cdots+d_m$, as nonzero. Further, since we have collected all the states with the corresponding columns having entries with index more than $d_1+\cdots+d_m$ as zero, the remaining states, if any, should have at least one of the entries with indices more than $d_1+\cdots+d_m$ as nonzero. This completes the induction step.
\end{proof}

Thus, for the rest of this article, we shall, without loss of generality, assume that the balanced replacement matrix $\bR$ with all entries nonnegative has diagonal blocks zero or irreducible and is in the increasing order.

\begin{rem} \label{rem: algo}
The arguments in the proof of Theorem~\ref{thm: algo} can be repeated when the replacement matrix is in the increasing order and it will be possible to identify the rate pairs more directly for the replacement matrices in the increasing order. Note that the arguments in the proof of Theorem~\ref{thm: algo} assumed that $\bQ_k$ is the scalar zero, whenever $\mu_k=0$, whereas the increasing order allows some of the leading zero diagonal blocks to be of higher dimension. However, we can break the leading zero diagonal blocks into the blocks of single colors and obtain the rate pairs separately. It is then easy to see that, if the $k$-th block is in the $j$-th cluster, that is, $i_j\le k<i_{j+1}$ for some $j=1,\ldots,J$ or $k=i_j$ for $j=J+1$, then $\alpha_k$ becomes the leading character of the cluster $\lambda_j$. Further, if the leading character of the cluster $\lambda_j$ or its order $\kappa_j$ is zero, then $\beta_k$ is the order of the leading character of the cluster, $\kappa_j$. It is not possible to identify $\beta_k$ in such a simple form, if both $\lambda_j$ and $\kappa_j$ are positive and we need to make the further assumption~\eqref{assmp}, which we are going to make in the next paragraph, to complete the identification. We obtain the complete identification of the rates and the limits in Theorem~\ref{thm: main} under the assumption~\eqref{assmp}. Yet, it would be important to note that much of the identification of the rate pairs for the replacement matrix in the increasing order in simple closed form is possible even without the assumption~\eqref{assmp}.
\end{rem}

To prove Theorem~\ref{thm: main}, as in \cite{bose:dasgupta:maulik:2009b}, we need to extend the condition~\eqref{def: incr nonlead} of Definition~\ref{def: incr} to the leading block of the $(j+1)$-th cluster, if the order of the leading character of the cluster is positive, that is $\kappa_{j+1}>0$, or equivalently $\lambda_j=\lambda_{j+1}$. In particular, we make the following assumption on a balanced, block upper triangular replacement matrix $\bR$ in the increasing order, formed by nonnegative entries, with blocks $\{\bQ_{ml}\}_{1\le m,l\le K+1}$, indices of the running maxima of the characters $\{i_j\}_{1\le j\le J+1}$, the leading characters $\{\lambda_j\}_{1\le j\le J+1}$ and their orders $\{\kappa_j\}_{1\le j\le J+1}$:
\let\myenumi\theenumi
\renewcommand{\theenumi}{\Alph{enumi}}
\begin{enumerate}
\item Whenever $\lambda_j>0$ and $\kappa_j>0$, we have $\sum_{m=i_{j-1}}^{i_j-1} \bQ_{mi_j} \neq \bzero$. \label{assmp}
\end{enumerate}
\renewcommand{\theenumi}{\myenumi}
This extension may not be possible in general. As a counterexample, consider the following upper triangular replacement matrix
$$\bR =
\begin{pmatrix}
0.5 & 0 & 0.5\\
0 & 0.5 & 0.5\\
0 & 0 & 1
\end{pmatrix}.$$
Note that we do not make the assumption~\eqref{assmp} when $\lambda_j=0$, as, by Lemma~\ref{lem: rearrange2}, this extension is possible if the character $0$ is repeated.

To reiterate, we shall assume, for the rest of the article, that the balanced replacement matrix $\bR$, with all entries nonnegative, has diagonal blocks zero or irreducible and is in the increasing order and further that the assumption~\eqref{assmp} holds. The assumption is made for the leading block of the $j$-th cluster for $j=1,\ldots,J+1$, whenever the leading character of the cluster is positive and the order of that leading character is also positive. Only~\eqref{assmp} has to be assumed, while the rest of the form can be obtained as a reduction from a general balanced replacement matrix with nonnegative entries using Lemmas~\ref{lem: rearrange} and~\ref{lem: rearrange2}.

For ease of understanding, we present below the $j$-th cluster, as a typical one, containing blocks with indices $i_j, i_j +1 , \ldots, i_{j+1}-1$ (note that the block with index $i_{j+1}$ actually goes to the next cluster):
$$\begin{pmatrix}
\bQ_{i_j} & \bQ_{i_j, i_j+1} & \cdots & \bQ_{i_j, k} & \cdots & \bQ_{i_j, i_{j+1}-1} \\
\bzero & \bQ_{i_j+1} & \cdots & \bQ_{i_j + 1, k} & \cdots & \bQ_{i_j+1, i_{j+1}-1} \\
\vdots & \ddots & \ddots & & \vdots & \vdots \\
\bzero & \cdots & \bzero & \bQ_k & \cdots & \bQ_{k, i_{j+1}-1} \\
\vdots & \vdots & \vdots & \ddots & \ddots & \vdots \\
\bzero & \cdots & \cdots & \cdots & \bzero & \bQ_{i_{j+1} - 1}
\end{pmatrix}.$$
The characters $\lambda_j = \mu_{i_j}, \mu_k, \lambda_{j+1} = \mu_{i_{j+1}}$ of the diagonal blocks $\bQ_{i_j}, \bQ_k, \bQ_{i_{j+1}}$ respectively, satisfy $\mu_k < \lambda_j \le \lambda_{j+1}$, for $i_j < k < i_{j+1}$. The index $\kappa_j$ counts number of times $\lambda_j$ has occurred as a character before $i_j$-th diagonal block.

Note that all the blocks have \PF eigenvalue less than or equal to $1$. As observed earlier, the last block is balanced with row sum $1$ and has the \PF eigenvalue, and thus the character, $1$ and hence forms the last cluster. If possible, suppose some earlier block has the character, and thus the \PF eigenvalue, $1$, then it should also be balanced with row sum $1$. Thus it will be a leading block and all other blocks in the same row will be zero submatrices. Hence, by the condition~\eqref{def: incr nonlead} of Definition~\ref{def: incr}, the immediately succeeding diagonal block cannot have the character less than $1$. So the next block will also have the character $1$ and will be a leading block and now the assumption~\eqref{assmp} will be violated for two successive leading blocks with the same leading character. Thus, for a balanced block triangular matrix with all entries nonnegative, which is in the increasing order and satisfies the assumption~\eqref{assmp}, except for the last block, no other block will have the character, and thus the \PF eigenvalue, $1$. Also, the last block will form the last cluster in itself.

When $\bQ_k$ is irreducible, the left and right eigenvectors of $\bQ_k$ corresponding to the \PF eigenvalue $\mu_k$ will be denoted as $\bpi^{(k)}$ and $\bzeta^{(k)}$, which are normalized so that $\bpi^{(k)} \bone = 1 = \bpi^{(k)} \bzeta^{(k)}$. Then, $\bpi^{(k)}$ and $\bzeta^{(k)}$ have all entries positive.

We now describe the main result to be presented in Theorem~\ref{thm: main}. Recall that the clusters with zero as the leading character are formed by a single block and are at the very beginning and the cluster with one as the leading character is the last one, that is, $(J+1)$-th one and is formed by the last, that is, $(K+1)$-th block. If the leading character of the $j$-th cluster, $\lambda_j=0$, then
$$\frac1{\log^{\kappa_j} N} \bC_N^{(i_j)} \to \frac1{j!} \bC_0^{(1)} \bQ_{12} \bQ_{23} \cdots \bQ_{i_j-1,i_j} \text{ almost surely and in $L^2$.}$$
If $i_j=1$, the continued product of the matrices above is interpreted as the identity matrix and as $\bC_0^{(1)}$ is assumed to be nonzero, by the condition~\eqref{def: incr zero} of Definition~\ref{def: incr}, we have the limit as a nonzero constant vector for all the clusters, if any, with zero leading character. For the last cluster, which contains only the (K+1)-th block, the scale is $N$ and the limit random vector is $\bpi^{(K+1)}$. In the $j$-th cluster, with the leading character $\lambda_j\in(0,1)$, corresponding to the leading block, we have a non-degenerate random variable $V_j$ such that
$$\frac1{N^{\lambda_j} \log^{\kappa_j} N} \bC_N^{(i_j)} \to V_j \bpi^{(i_j)} \text{ almost surely and in $L^2$.}$$
Recall that, since the leading character $\lambda_j$ is nonzero, the submatrix $\bQ_{i_j}$ is nonzero and irreducible and has the \PF eigenvalue $\lambda_j$ and the corresponding left eigenvector $\bpi^{(i_j)}$ normalized so that the sum of the entries is one. For other blocks in the $j$-th cluster, we use the same scaling and the limit random vector is again constant vector multiple of $V_j$, where the constant vector is obtained by multiplying $\bpi^{(i_j)}$ on right by a constant matrix. Further, if $\kappa_j > 0$, then $V_j$ is a constant non-zero multiple of $V_{j-1}$. The matrix and the scalar multiples have been defined in~\eqref{def: nonlead w} and~\eqref{def: w} respectively. Thus, within a cluster, the scale remains same. Further, in the clusters where the leading characters are same, the scales change by powers of $\log N$, but the limiting random vectors continue to be constant vector multiples of one scalar random variable. The random variable is degenerate, if the leading character is zero or one. Thus, the number of random variables involved in the limit equals the number of distinct leading characters $\lambda_j$ in the open interval $(0,1)$.

In the triangular case discussed in \cite{bose:dasgupta:maulik:2009b}, the scaled count of the leading color of a block converged to a random variable. In the block triangular case, we analogously consider the scaled color count subvector corresponding to the leading block of a cluster. Here the limiting vector still has one dimensional randomness, as it is a constant vector multiple of a scalar valued random variable. Further, in the triangular case, the scaled count subvector for a block converged to a vector which is a constant vector multiple of the limit random variable corresponding to the leading color. Analogously, the scaled count vector for a cluster in the block triangular case converges to a vector which is a constant matrix multiple of the limit vector corresponding to the leading block. Moreover, in the triangular case, if the diagonal entries corresponding to two successive leading colors are same, then the corresponding limiting random variables are multiples of each other. However, in the block triangular case, the limit vectors corresponding to two successive leading blocks with the same \PF eigenvalues are not scalar or matrix multiples of each other. Yet, the scalar random variables associated with the limit vectors are multiples of each other.
\end{section}

\begin{section}{Identification of the limits} \label{sec: main}
For the purpose of this section, the replacement matrix $\bR$ has nonnegative entries, is balanced, block upper triangular, in the increasing order and satisfies the assumption~\eqref{assmp}. The notations for the blocks, the characters, the leading characters and their orders, the eigenvalues and the eigenvectors remain the same.

To identify the limits, we need to define a sequence of matrices $\{\bW_k\}_{k=1}^{K+1}$ corresponding to the blocks and another sequence of constants $\{w_j\}_{j=1}^{J+1}$ corresponding to the clusters. Actually, these matrices and constants are useful only when the corresponding characters and leading characters are positive, but we define them in all the cases. We define them inductively. We first define the matrices $\bW_k$'s. For the definition of $\{\bW_k\}$, we only require the replacement matrix $\bR$ to be in the increasing order but we do not need to assume~\eqref{assmp}, as follows:
\begin{enumerate}
\item If $k$ corresponds to a leading block, that is, $k=i_j$ for some $j=1,\ldots,J+1$ (this includes the cases $k=1$ and $k=K+1$), define $\bW_k=\bI$, the identity matrix of order $d_k$, where, recall that $d_k$ is the number of colors in the $k$-th block.
\item If $k$ corresponds to a non-leading block, that is, $i_j < k < i_{j+1}$ for some $j=1,\ldots,J$, define
\begin{equation} \label{def: nonlead w}
\bW_k = \sum_{m=i_{j}}^{k-1} \bW_m \bQ_{mk} \left( \lambda_j \bI - \bQ_k \right)^{-1},
\end{equation}
a matrix of order $d_{i_j} \times d_k$.
\end{enumerate}
Note that, for $i_j < k < i_{j+1}$, $\mu_k$, the character, and hence the \PF eigenvalue, of $\bQ_k$ satisfies $\mu_k < \lambda_j$. Hence the absolute values of all the eigenvalues of $\bQ_k$, which are smaller than or equal to $\mu_k$, are strictly smaller than $\lambda_j$, making $(\lambda_j \bI - \bQ_k)$ invertible. Further,
$$\left( \lambda_j \bI - \bQ_k \right)^{-1} = \frac1{\lambda_j} \left[ \bI + \sum_{i=1}^\infty \left( \frac1\lambda_j \bQ_k \right)^i \right],$$ where the sum on the right side converges. Observe that all the matrices in the summation on the right side have all entries nonnegative. Further, if $\mu_k>0$, or equivalently, $\bQ_k$ is irreducible, for each element of the matrix on the left side, some power of $\bQ_k$ has the corresponding element strictly positive. Thus, $( \lambda_j \bI - \bQ_k )^{-1}$ has all elements strictly positive, whenever $\mu_k>0$. If $\mu_k=0$, or equivalently, $\bQ_k$ is the scalar zero, then $(\lambda_j\bI-\bQ_k)^{-1}=1/\lambda_j$ is a finite positive number too, as $\lambda_j>0$.

The constants $w_j$'s are defined, when the replacement matrix $\bR$ is in the increasing order and satisfies the assumption~\eqref{assmp}, as follows:
\begin{equation} \label{def: w}
w_j =
\begin{cases}
  {\displaystyle 1}, &\text{if $\lambda_j = 0$ or $\kappa_j = 0$,}\\
  {\displaystyle \frac1{\kappa_j} \bpi^{(i_{j-1})} \sum_{m=i_{j-1}}^{i_j-1} \bW_m \bQ_{mi_j} \bzeta^{(i_j)}}, &\text{otherwise.}
\end{cases}
\end{equation}
Note that if $\kappa_j>0$ and $\lambda_j>0$, then $\lambda_j=\lambda_{j-1}>0$ as well. Thus, $\bQ_{i_{j-1}}$ and $\bQ_{i_j}$ are both irreducible and it is meaningful to define the left and the right eigenvectors $\bpi^{(i_{j-1})}$ and $\bzeta^{(i_j)}$ respectively corresponding to the respective \PF eigenvalues $\lambda_{j-1}$ and $\lambda_j$.

Next proposition shows that right multiplication of a vector with all coordinates positive by $\bW_k$ keeps all coordinates of the resultant vector positive. This is important to show that the limit random vector obtained by scaling the count vectors have all coordinates non-degenerate.

\begin{prop} \label{prop: w}
Let $\bpi$ be a vector with all coordinates strictly positive and the replacement matrix $\bR$ be in the increasing order. Then, for all $1\le k\le K+1$, the vector $\bpi \bW_k$ has all coordinates strictly positive.
\end{prop}
\begin{proof}
The proof is done through induction. For $k=1=i_1$, we have $\bW_1 = \bI$ and hence $\bpi \bW_1 = \bpi$ has all coordinates strictly positive. Then, assume $\bpi \bW_m>0$ for all $m<k$. We consider two cases of $k$ separately.

For $k=i_j$ for some $j=1,\ldots,J+1$, by definition $\bW_{i_j}=\bI$ and $\bpi \bW_{i_j} = \bpi$ has all coordinates strictly positive.

Next consider $k$ such that $i_j < k < i_{j+1}$ for some $j=1,\ldots,J$. Then, by~\eqref{def: nonlead w},
$$\bpi \bW_k = \sum_{m=i_{j}}^{k-1} \bpi \bW_m \bQ_{mk} \left( \lambda_j \bI - \bQ_k \right)^{-1}.$$
By induction hypothesis, for all $i_j\le m<k$, $\bpi \bW_m$ have all coordinates strictly positive. Next, we consider two subcases separately according as $\mu_k$ is zero or positive.

First, we assume $\mu_k=0$. Since we are considering a non-leading block, the diagonal block must be the scalar zero. Then $\bpi \bW_k = \lambda_j^{-1} \sum_{m=i_j}^{k-1} \bpi \bW_m \bQ_{mk}$, where $\bQ_{mk}$ are $d_m$-dimensional column vectors. Since, $\bpi \bW_m$ have all coordinates positive, for all $i_j\le m<k$, and by the condition~\eqref{def: incr nonlead} in Definition~\ref{def: incr}, at least one of the column vectors $\bQ_{mk}$ with $i_j\le m<k$ must be nonzero, we must have $\bpi \bW_k$ is a nonzero scalar.

Next, we assume $\bQ_k$ is irreducible and then we have observed that $( \lambda_j \bI - \bQ_k )^{-1}$ has all elements strictly positive. Also, by the condition~\eqref{def: incr nonlead} of Definition~\ref{def: incr}, we have, for some $i_j \le m < k$, $\bQ_{mk}$ must have at least one element strictly positive. Thus, for that $m$, $\bpi \bW_m \bQ_{mk} ( \lambda_j \bI - \bQ_k )^{-1}$ will have all coordinates strictly positive. All other summands on the right side have all coordinates nonnegative. Thus $\bpi \bW_k$ has all coordinates positive.

This completes the induction step.
\end{proof}

We next show that the constants $w_j$'s are positive. This will show that for a cluster with the \PF eigenvalue of the leading block same as that of a previous cluster, the limit random vector obtained from the scaled color counts is non-degenerate.

\begin{cor} \label{cor: w}
Let the replacement matrix $\bR$ be in the increasing order and satisfy the assumption~\eqref{assmp}. Then, for all $1\le j\le J+1$, we have $w_j > 0$.
\end{cor}
\begin{proof}
If $\kappa_j=0$ or $\lambda_j=0$, then $w_j=1$ and the result holds. Thus, without loss of generality, we can assume $\kappa_j > 0$ and $\lambda_j>0$. Since $\bpi^{(i_{j-1})}$ is a left \PF eigenvector of an irreducible matrix $\bQ_{i_{j-1}}$, it has all coordinates positive. So, by Proposition~\ref{prop: w}, for all $i_{j-1} \le m < i_j$, $\bpi^{(i_{j-1})} \bW_m$ has all coordinates positive. Since $\bzeta^{(i_j)}$ is the right \PF eigenvector of an irreducible matrix $\bQ_{i_j}$, it again has all coordinates positive. Since $\kappa_j>0$, by the assumption~\eqref{assmp}, we have, for some $i_{j-1} \le m < i_j$, $\bQ_{mi_j}$ must have at least one element strictly positive. Thus, for that $m$, $\bpi^{(i_{j-1})} \bW_m \bQ_{mi_j} \bzeta^{(i_j)} > 0$. Also, for $i_{j-1} \le m < i_j$, $\bQ_{mi_j}$ have all elements nonnegative. Hence, $w_j > 0$.

This completes the induction step.
\end{proof}

Now, we are ready to identify the limits.
\begin{thm} \label{thm: main}
Consider an urn model with a balanced, block upper triangular replacement matrix $\bR$, formed by nonnegative entries, having $K+1$ blocks, which is in the increasing order with $J+1$ clusters having the leading characters $\{\lambda_j\}_{1\le j\le J+1}$ and their orders $\{\kappa_j\}_{1\le j\le J+1}$. We further assume that $\bR$ satisfies the assumption~\eqref{assmp}. Then the color count subvector corresponding to the $k$-th block, for $k=1, \ldots, K$, satisfying $i_j \le k < i_{j+1}$, for some $j=1, \ldots, J$, we have,
$$\frac1{N^{\lambda_j} \log^{\kappa_j} N} {\bC_N^{(k)}} \to
\begin{cases}
{\displaystyle \frac1{\kappa_j!} \bC_0^{(1)} \prod_{m=1}^{i_{j-1}} \bQ_{m,m+1}}, &\text{if $\lambda_j=0$}\\
{\displaystyle V_j \bpi^{(i_j)} \bW_k}, &\text{if $\lambda_j>0$}
\end{cases}$$
almost surely and in $L^2$, where $V_j$ is a nondegenerate random variable for $j\le J$. For $\lambda_1=0$, the continued product in the limit above is interpreted as $\bI$.

Furthermore, for the random variables $V_j$, $j=1, \ldots, J$, if $\kappa_j>0$, we also have
\begin{equation} \label{eq: v recursion}
V_j = w_j V_{j-1}.
\end{equation}

Finally, for the last count subvector we have
$$\frac1N \bC_N^{(K+1)} \to \bpi^{(K+1)} \quad \text{almost surely and in $L^2$.}$$
\end{thm}
\begin{proof}
Note that the result for the $(K+1)$-th block is in \cite{gouet:1997}. The rest we prove by induction on $k$. Much of the argument has already been completed during the proof of Theorem~\ref{thm: algo}. We shall now use the special structure of the replacement matrix in the increasing order and the assumption~\eqref{assmp} to identify the limits. The rate pair $(\alpha,\beta)$ will mean same as in Theorem~\ref{thm: algo}.

For $k=1$, we are in the leading block of the first cluster. If $\lambda_1=\mu_1=0$, then clearly $\bC_N^{(1)}=\bC_0^{(1)}$ and we have the required limit.  For $\lambda_1=\mu_1>0$, we use the argument given in the proof of Theorem~\ref{thm: algo} and by choosing $\bzeta=\bzeta^{(1)}$ as the normalized right eigenvector, we obtain from~\eqref{eq: lim one}, $\bC_N^{(1)}/N^{\mu_1}\to V_1\bpi^{(1)}$ almost surely and in $L^2$, for the nondegenerate random variable $V_1=Y_1$, as required.

Next, we assume that the result holds for all blocks till $(k-1)$-th block and we study the $k$-th block. We first identify the rate pairs as obtained in Theorem~\ref{thm: algo} using the fact that the replacement matrix is in the increasing order and the assumption~\eqref{assmp} has been made. Recall that an intermediate rate pair $(\alpha,\beta)$ was defined as the lexicographically largest one in the set $\{(\alpha_m,\beta_m): 1\le m\le k-1, \bQ_{mk}\neq\bzero\}$. Then, we obtained $\alpha_k=\max\{\alpha,\mu_k\}$ and $\beta_k$ equaled $0$, $\beta+1$ or $\beta$ according as $\mu_k$ is larger than, equal to or smaller than $\alpha$. We consider three cases separately, namely, $k$-th block is the leading one of its cluster with the corresponding order $0$ in the first case, positive in the second case and the block is a nonleading one of its cluster in the third case. In the first case $k=i_j$ for some $j=1,\ldots,J+1$ and $\kappa_j=0$. Thus, for all $1\le m\le k-1$, we have $\mu_m<\mu_k$ and hence for all $j^\prime=1,\ldots,j-1$, we have $\lambda_{j^\prime}<\mu_k$. So, by the induction hypothesis, for all $m=1,\ldots, k-1$, $(\alpha_m,\beta_m)$ are, and as a consequence, $(\alpha,\beta)$ is lexicographically strictly smaller than $(\mu_k,0)$. This gives $\alpha<\mu_k$ and $(\alpha_k,\beta_k) = (\mu_k,0) = (\lambda_j,\kappa_j)$. In the second case $k=i_j$ for some $j=2,\ldots,J$ and $\kappa_j>0$. (Note that $\kappa_1=0$ and the last $(J+1)$-th cluster is the only cluster with the leading character $1$ and hence $\kappa_{J+1}=0$.) Also $\mu_k=\lambda_j=\lambda_{j-1}$ and $\kappa_{j-1}+1=\kappa_j$ in this case. If $\mu_k=\lambda_j=0$, then using the condition~\eqref{def: incr zero} of Definition~\ref{def: incr}, and, if $\mu_k=\lambda_j>0$, then using the assumption~\eqref{assmp}, we have $\bQ_{mi_j}\neq\bzero$ for some $m=i_{j-1},\ldots,i_j-1$. Then, using the induction hypothesis, the (lexicographically) largest rate pair $(\alpha,\beta)$ is attained at such a value of $m$ and $(\alpha,\beta) = (\lambda_{j-1},\kappa_{j-1}) = (\lambda_j,\kappa_j-1)$. Hence $\alpha=\mu_k$ and $(\alpha_k,\beta_k)=(\lambda_j,\kappa_k)$. For the third case, $i_j<k<i_{j+1}$ for some $j=1,\ldots,J$. By the condition~\eqref{def: incr nonlead} of Definition~\ref{def: incr}, we have $\bQ_{mk}\neq\bzero$ for some $m=i_j,\ldots,k-1$. Then, using the induction hypothesis, the (lexicographically) largest rate pair $(\alpha,\beta)$ is attained at such a value of $m$ and $(\alpha,\beta) = (\lambda_j,\kappa_j)$. Since the $k$-th block is nonleading, we have $\mu_k<\lambda_j$ and hence $\alpha>\mu_k$ and $(\alpha_k,\beta_k)=(\lambda_j,\kappa_j)$. Thus, considering all the three cases, we have that the rate for the $k$-th block, which is in the $j$-th cluster, is $(\alpha_k,\beta_k)=(\lambda_j,\kappa_j)$. Also the three cases that $\alpha$ is less than, equal to or greater than $\mu_k$ are equivalent to the cases that the $k$-th block is the leading block of its cluster with the order of the leading character positive, the same with the order of the leading character zero or the $k$-th block is a nonleading block of its cluster respectively.

Next, we identify the limits by simplifying the results obtained during the proof of Theorem~\ref{thm: algo} using the fact that the replacement matrix is in the increasing order and the assumption~\eqref{assmp} has been made. First observe that the situation that $\alpha_k=-\infty$ and $\mu_k=0$ require that $\bQ_{mk}=\bzero$ for all $m=1,\ldots,k$ and hence for all $m$, as the matrix is block upper triangular. However such colors have been included in the first block by the construction in the proof of Lemma~\ref{lem: rearrange2} and this situation does not occur for $k>1$. So we can assume~\eqref{eq: nonzero} that $\bQ_{mk}\neq\bzero$ for some $m=1,\ldots,k$.

First consider the case $\mu_k=0$. We shall assume that the $k$-th block is in the $j$-th cluster, for some $j=1,\ldots,J$. As in the limit in~\eqref{eq: lim zero}, we have two subcases, namely, $\alpha=\mu_k=0$ and $\alpha>\mu_k=0$. Recall that the subcase $\alpha=\mu_k=0$ is equivalent to the fact that the $k$-th block is the leading block of the $j$-th cluster with the order of the leading character being positive. Thus the leading character is $0$ and the $k$-th block is actually the $k$-th initial cluster with zero character and then $\kappa_k=k-1>0$. Also, using the induction hypothesis and the condition~\eqref{def: incr zero} in Definition~\ref{def: incr}, only $m=k-1$ contributes to the limit in~\eqref{eq: lim zero} and it simplifies to
$$\frac1{k-1} \frac1{(k-2)!} \bC_0^{(1)} \prod_{m=1}^{i_{k-2}} \bQ_{m,m+1} \bQ_{k-1,k} = \frac1{\kappa_k!} \bC_0^{(1)} \prod_{m=1}^{i_{k-1}} \bQ_{m,m+1}.$$
The subcase $\alpha>\mu_k=0$ is equivalent to the fact that the block is a non-leading block of the $j$-th cluster and, by the condition~\eqref{def: incr scalar} of Definition~\ref{def: incr}, contains a single color. Also the diagonal block is the scalar zero, making $\lambda_j\bI-\bQ_k$ the scalar $\lambda_j$. Further, by the induction hypothesis and the condition~\eqref{def: incr nonlead} of Definition~\ref{def: incr}, only $m=i_j,\ldots, k-1$ contributes to the limit in~\eqref{eq: lim zero}. Note that some of the corresponding $\bQ_{mk}$ may be the zero matrix, but the corresponding extra terms will not affect the limit. So the limit simplifies to, using~\eqref{def: nonlead w},
$$\frac1{\lambda_j} V_j \bpi^{(i_j)} \sum_{m=i_j}^{k-1} \bW_m \bQ_{mk} = \frac1{\lambda_j} V_j \bpi^{(i_j)} \bW_k (\lambda_j\bI-\bQ_k) = V_j \bpi^{(i_j)} \bW_k.$$

Finally consider the case $\mu_k>0$. We identify the limit for the blocks with positive character from~\eqref{eq: lim blk k}. We shall consider three subcases $\mu_k>\alpha$, $\mu_k<\alpha$ and $\mu_k=\alpha$. The subcase $\mu_k>\alpha$ is equivalent to the fact that the $k$-th block is the leading block of the $j$-th cluster for some $j=1,\ldots,J+1$ with $\kappa_j=0$. Then, using $\bzeta=\bzeta^{(i_j)}$ and hence $\bpi=\bpi^{(i_j)}$ in the argument of the proof of Theorem~\ref{thm: algo}, we have the limit in the required form, where we declare $V_j=Y_{i_j}$ and use the fact that $\bW_{i_j}=\bI$. The second subcase $\mu_k<\alpha$is equivalent to the fact that the $k$-th block is a nonleading block of its cluster. Also, by the induction hypothesis, the limit vectors $\bu_m$ in~\eqref{eq: lim blk k} will be nonzero only for $m=i_j,\ldots,k-1$. If some $m=i_j,\ldots,k-1$ has $\bQ_{mk}=\bzero$, they still can be included in the sum in~\eqref{eq: lim blk k}, as they all contribute zero vectors. Then, by the induction hypothesis, the fact that $\alpha_k=\lambda_j$ and~\eqref{def: nonlead w}, the limit becomes
$$V_j \bpi^{(i_j)} \sum_{m=i_j}^{k-1} \bW_m \bQ_{mk} (\lambda_j\bI-\bQ_k)^{-1} = V_j \bpi^{(i_j)} \bW_k.$$
At the end, we consider the subcase $\mu_k=\alpha$, which is equivalent to the fact that the $k$-th block is the leading block of its cluster with the order of the leading character being positive. Thus, we have $k=i_j$ for some $j=1,\ldots,J$ and $\kappa_j>0$. As in the subcase $\mu_k>\alpha$, by considering $\bzeta=\bzeta^{(i_j)}$ and hence $\bpi=\bpi^{(i_j)}$, together with the fact that $\bW_{i_j}=\bI$ and denoting $V_j=Y_{i_j}=Y_k$, we have, from~\eqref{eq: lim blk k}, the limit as $V_j \bpi^{(i_j)} \bW_{i_j}$. So, to complete the proof, we only need to check~\eqref{eq: v recursion}, which is done by simplifying~\eqref{eq: lim positive}. Observe that for the subcase $\mu_k=\alpha$, we have $\alpha=\alpha_k=\mu_k=\lambda_j$ and $\beta_k=\beta+1=\kappa_j=\kappa_{j-1}+1$. Then, by the assumption~\eqref{assmp} and the induction hypothesis, only the terms corresponding to $m=i_{j-1},\ldots,k-1$ will contribute to the limit in~\eqref{eq: lim positive}. Also, the terms corresponding to $m=i_{j-1},\ldots,k-1$ with $\bQ_{mk}=\bzero$ can be included in the sum in the limit, as they do not contribute anything extra. Thus, recalling the facts that $\bzeta=\bzeta^{(i_j)}$ and $k=i_j$ and using~\eqref{def: w}, the limit in~\eqref{eq: lim positive} simplifies to
$$Y_k=V_j= \frac1{\kappa_j} V_{j-1} \bpi^{(i_{j-1})} \sum_{m=i_{j-1}}^{k-1} \bW_m \bQ_{mi_j} \bzeta^{(i_j)} = w_j V_{j-1}$$
and proves~\eqref{eq: v recursion}.
\end{proof}
\end{section}

\textbf{Acknowledgement:} The authors like to thank an anonymous referee for careful reading of the manuscript and suggesting improvements.

\end{document}